\documentclass[12pt]{amsart}

\usepackage{amssymb, amsmath, xcolor}
\usepackage{multirow}
\usepackage{import}
\usepackage{tikz}
\usetikzlibrary{shapes}
\tikzset{filled/.style={minimum width=5pt,inner sep=0pt,circle,fill=black}}
\usepackage{subcaption}
\usepackage{mathtools}
\usepackage[utf8]{inputenc}
\usepackage{float}
\usepackage{todonotes}
\usepackage[square,numbers]{natbib}
\usepackage{qtree}
\usepackage{adjustbox}

\newtheorem{theorem}{Theorem}[section]
\newtheorem{lemma}[theorem]{Lemma}
\newtheorem{corollary}[theorem]{Corollary}
\newtheorem{conjecture}[theorem]{Conjecture}
\newtheorem{question}[theorem]{Question}

\theoremstyle{definition}
\newtheorem{definition}[theorem]{Definition}
\newtheorem{example}[theorem]{Example}

\theoremstyle{remark}
\newtheorem{remark}[theorem]{Remark}

\numberwithin{equation}{section}
\numberwithin{figure}{section}

\renewcommand{\mod}{\operatorname{mod}}
\newcommand{\M}{\operatorname{M}}
\newcommand{\FM}{\operatorname{FM}}
\newcommand{\diam}{\operatorname{diam}}
\newcommand{\N}{\mathbb{N}}
\newcommand{\Z}{\mathbb{Z}}

\title[$2$-Distance Graphs]{Limits and Periodicity of Metamour $2$-Distance Graphs} 

\author[Erickson, Herden, Meddaugh, Sepanski, $\ldots$]{William Q. Erickson, Daniel Herden, Jonathan Meddaugh, Mark R. Sepanski, Mitchell Minyard, Kyle Rosengartner}

\address{
All authors:
Department of Mathematics,
Baylor University,
Sid Richardson Building,
1410 S.~4th Street,
Waco, TX 76706, USA}

\email{will\_erickson@baylor.edu, daniel\_herden@baylor.edu,  jonathan\_meddaugh@baylor.edu, mark\_sepanski@baylor.edu, kyle\_rosengartner1@baylor.edu}

\date{\today}

\begin{document}


\begin{abstract}
    Given a finite simple graph $G$, let $\M(G)$ denote its 2-distance graph, in which two vertices are adjacent if and only if they have distance 2 in $G$.  
    In this paper, we consider the periodic behavior of the sequence $G, \M(G), \M^2(G), \M^3(G), \ldots$ obtained by iterating the 2-distance operation.
    In particular, we classify the connected graphs with period 3, and we partially characterize those with period 2.
    We then study two families of graphs whose 2-distance sequence is \emph{eventually} periodic: namely, generalized Petersen graphs and complete $m$-ary trees.
    For each family, we show that the eventual period is 2, and we determine the pre-period and the two limit graphs of the sequence.
\end{abstract}

\maketitle

\tableofcontents

\section{Introduction}

The notion of an \emph{$n$-distance graph} was introduced by Harary--Hoede--Kadlecek~\cite{HHK}, and defined as follows: given a finite simple graph $G$, its $n$-distance graph is obtained by placing an edge between two vertices if and only if those vertices have distance $n$ in the original graph $G$.
Since then, the study of $n$-distance graphs (in particular, the special case $n=2$) has developed in several directions, and has garnered increased interest within the last decade.
This recent interest includes work on the connectivity of $2$-distance graphs \cite{JM23, Khormali}, their regularity~\cite{GaarKrenn2023}, their diameter~\cite{JM24}, $2$-distance graphs which have certain maximum degree~\cite{AG1} or are isomorphic to the original graph~\cite{AG2}, and general characterizations of certain $2$-distance graphs~\cite{ChingGarces19}.
Another topic of interest is the periodicity of graphs with respect to the $2$-distance operation.
This question was first addressed in the note~\cite{Zelinka} in 2000 (see also~\cite{Leonor}), and is the subject of the present paper.

Throughout the paper, we adopt the term \emph{metamour graph}, which was recently introduced in~\cite{GaarKrenn2023} as a synonym for the 2-distance graph.
Starting with some finite simple graph $G$, we consider the sequence $G, \M(G), \M^2(G), \M^3(G), \ldots$, obtained by repeatedly taking metamour graphs.
A natural problem is to describe the periodic behavior of this metamour sequence, in as much generality as possible.
Our paper solves this problem in several special cases, which we highlight below:
\begin{itemize}
    \item In Section \ref{pseud1}, we discuss graphs $G$ with the property $M(G)=\overline{G}\cong G$. In particular, we provide two proofs that every graph is an induced subgraph of some graph $G$ with $M(G)=\overline{G}\cong G$, see Theorems \ref{thm: paley} and \ref{thm: selfcomp}.
    \item In Theorem~\ref{thm: odd C_n with decorations}, we determine the periodic behavior of the metamour sequence in the case where $G$ is formed by joining an arbitrary number of graphs together in a cycle.
    (This generalizes Proposition~2 in Zelinka \cite{Zelinka}.)
    \item Theorem~\ref{thm:induced subgraph} states that for any even positive integer $k$, every graph is an induced subgraph of a graph whose metamour sequence is $k$-periodic.
    (This theorem strengthens the main theorem in Zelinka~\cite[p.~268]{Zelinka}, which merely asserted that $k$-periodic graphs exist for each even $k$.)
    \item In Theorem~\ref{thm: period 2 metamour diameter constraints}, we show that if $G$ has metamour period 2, then $G$ and $\M(G)$ have the same diameter, which is at most $3$.
    \item In fact, when this diameter equals 3, we identify a certain subgraph $\widehat{C}_5$ (see Theorem~\ref{thm: C5 with party hat}) which must be contained in $G$.
    \item In Theorem~\ref{thm: classification period 3}, we prove that the only connected graphs with metamour period 3 are the two cycle graphs $C_7$ and $C_9$.
\end{itemize}

In the final two sections of the paper, we consider graphs whose metamour sequence is \emph{eventually} periodic.
In particular, we study two fundamental families in graph theory: the generalized Petersen graphs $G(m,2)$, and the complete $m$-ary trees.
In each case (see Theorems~\ref{thm: Peteresen G(m,2) limit sets} and~\ref{cor:tree limits}), we show that the eventual period is 2, we determine where in the sequence $G, \M(G), \M^2(G), \M^3(G), \ldots$ this periodicity begins, and we explicitly describe the two graphs that are the limits of this metamour sequence.

In addition to our main results summarized above, we also answer (in Theorem~\ref{thm: inverse image of M}) a question posed in~\cite[Question~10.5]{GaarKrenn2023}, regarding the classification of metamour graphs. 
Moreover, the work in this paper led us to Questions \ref{q1} and \ref{q2}, and to
Conjecture~\ref{conj: odd periodic}, regarding the parity of metamour periods: roughly speaking, graphs with even metamour periods are ``common'' (as demonstrated in Theorem~\ref{thm:induced subgraph}) and those with odd metamour periods are ``rare'' (conjecturally, finitely many).
Hence a natural direction for further research is the problem of proving Conjecture~\ref{conj: odd periodic}, and classifying those graphs whose metamour period is even (respectively, odd).

\section{Notation}

We write $\N$ for the nonnegative integers and $\Z^+$ for the positive integers. 
Throughout the paper, we use standard graph theoretical notation and terminology, which we summarize as follows.
We write $G=(V,E)$ to denote a simple graph with vertex set $V$ and edge set $E$.
We also use $V(G)$ and $E(G)$ to denote the vertex and edge sets of $G$.
We abbreviate the edge $\{x,y\}$ by writing $xy$.
Let 
$d_G:V\times V\rightarrow \N \cup \{\infty\}$
be the distance function on $G$. 
Recall that the \emph{diameter} of a connected graph $G$, denoted by $\diam(G)$, is the maximum value attained by $d_G$. A disconnected graph has infinite diameter.
We write $G_1 \cup G_2$ for the \emph{union} of two graphs, in which the vertex set is the disjoint union $V(G_1) \cup V(G_2)$ and the edge set is the disjoint union $E(G_1) \cup E(G_2)$.
We write $G_1 \nabla G_2$ for the \emph{join} of two graphs, which is $G_1 \cup G_2$ together with all edges connecting $V(G_1)$ and $V(G_2)$.
We write $\overline{G}$ for the \emph{complement} of $G$, where $V(\overline{G})=V(G)$, and $xy \in E(\overline{G})$ if and only if $xy \not\in E(G)$.
We write $K_n$ for the complete graph on $n$ vertices, and $C_n$ for the cycle graph on $n$ vertices.
The \emph{edgeless graph} on $n$ vertices is the graph with no edges, which we denote by $\overline{K}_n$.
We write $G_1 \subseteq G_2$ to express that $G_1$ is a subgraph of $G_2$.
Given a subset $S \subseteq V(G)$, the \emph{induced subgraph} is the graph with vertex set $S$, whose edges are all the edges in $E(G)$ with both endpoints in $S$.
We write $G_1 = G_2$ to denote equality as graphs of distinguishable vertices, and we write $G_1 \cong G_2$ to express that two graphs are isomorphic (i.e., equal up to forgetting vertex labels). A graph $G$ is called \emph{self-complementary} if $G\cong \overline{G}$.

The focus of this paper is the \emph{metamour graph} of $G$ (also known as the \emph{2-distance graph} in the literature):

\begin{definition}\label{def: M}
    The \emph{metamour graph} of $G$, written as
        \[ \M(G),\]
    has vertex set $V(G)$, and an edge between $x,y\in V(G)$ if and only if $d_G(x,y) = 2$. 

    More generally, we write $\M^0(G) \coloneqq G$ and inductively define
        \[ \M^{k+1}(G) \coloneqq \M(\M^{k}(G)), \qquad k \in \N. \]
\end{definition}

\begin{definition}
\label{def:intro}
    Let $k\in\Z^+$. 
    We say that $G$ has \emph{metamour period $k$} if
    \[ \M^k(G) = G\]
    and $k$ is minimal with this property.
    We say that $G$ has \emph{pseudo-metamour period $k$} if
    \[ \M^k(G) \cong G \]
    and $k$ is minimal with this property.
    We say that $G$ is \emph{metamour-complementary} if
    \[ \M(G) = \overline{G}.\]
\end{definition}

Whereas Definition~\ref{def:intro} above pertains to periodic sequences of metamour graphs, the following definition concerns those metamour sequences which are \emph{eventually} periodic:

\begin{definition}
    Let $k\in\Z^+$. We say that $G$ has \emph{metamour limit period~$k$} if there exists $N \in \N$ such that, for all integers $i \geq N$, we have
    \[ \M^{k+i}(G) = \M^i(G), \]
    and $k$ is minimal with this property. In this case, we write
    \[ \lim \M(G) \coloneqq \{ \M^i(G) \mid  i \geq N\} =
        \{ \M^N(G), \M^{N+1}(G), \ldots, \M^{N+k-1}(G) \} \]
    for the \emph{metamour limit set} of $G$.
\end{definition}

Note that for finite graphs, the metamour limit period always exists, and so the metamour limit set is also finite.

\section{Graphs That Are Metamour Graphs}

In this section, we collect some basic results on metamour graphs and metamour-complementary graphs. Our main goal is a characterization and discussion of those graphs $G$ such that $G=\M(G')$ for some graph $G'$, i.e., of graphs which are metamour graphs (see Theorem~\ref{thm: inverse image of M}).

We start by noting the following relation between $\M(G)$ and $\overline{G}$.

\begin{lemma} \label{lem: subset of G bar}
    For any graph $G$, we have
    \[ \M(G)\subseteq \overline{G}.\]
    If there is a graph $G'$ on $V(G)$ such that $\M(G')=G$, then $G'\subseteq \overline{G}$.
\end{lemma}

\begin{proof}
    This follows immediately from Definition \ref{def: M}. In particular, if $xy \in E(G)$, then $d_G(x,y) = 1$, thus $xy \notin E(\overline{G})$.
\end{proof}

Metamour-complementary graphs are defined as the graphs $G$ where equality is achieved in $\M(G)\subseteq \overline{G}$. We provide an alternative characterization of metamour-complementary graphs.

\begin{theorem}
    A graph $G$ is metamour-complementary if and only if $\diam(G) \leq 2$.
    \label{M(G) = G^c}
\end{theorem}

\begin{proof}
    Suppose $\M(G) = \overline{G}$. If $x,y \in V$ with $d_G(x,y) \geq 2$, then $xy \in E(\overline{G}) = E(\M(G))$ and so $d_G(x,y) = 2$.

    Now suppose $\diam(G) \leq 2$. Then for $xy \in E(\overline{G})$, we have $d_G(x,y) = 2$ and so $xy \in E(\M(G))$. Since $\M(G) \subseteq \overline{G}$ by Lemma \ref{lem: subset of G bar}, we are done.
\end{proof}

Theorem \ref{M(G) = G^c} can be expanded to provide an answer to \cite[Question~10.5]{GaarKrenn2023} by characterizing those graphs $G$ which are metamour graphs.

\begin{theorem} \label{thm: inverse image of M}
    Let $G=(V,E)$ be a graph. The following are equivalent:
    \begin{enumerate}
        \item There exists a graph $G'$ on $V$ such that $\M(G')=G$.
        \item $\overline{G}$ is metamour-complementary, i.e., $\M(\overline{G})=G$.
        \item For every $xy \in E$, there exists $z \in V\setminus \{x,y\}$ with $xz,yz \notin E$.
        \item $\diam(\overline{G}) \leq 2$.
    \end{enumerate}
\end{theorem}

\begin{proof}
    We first show that (1) and (2) are equivalent.
    It is immediate that (2) implies (1).
    To show that (1) implies (2), suppose $\M(\overline{G}) \neq G$. 
    By Definition \ref{def: M}, there exists $xy \in E(G)$ such that there is no $z \in V \setminus \{x,y\}$ with $xz,yz \in E(\overline{G})$. 
    Therefore, for any $G'\subseteq \overline{G}$, we have $xy\not\in \M(G')$. 
    Since (by Lemma \ref{lem: subset of G bar}) any $G'$ satisfying $\M(G')=G$ satisfies $G'\subseteq \overline{G}$, we are done.

    The equivalence of (2) and (3) follow immediately from Definition~\ref{def: M}, and the equivalence of (2) and (4) from Theorem \ref{M(G) = G^c}.    
\end{proof}





We also note the following nice sufficient characterization.

\begin{corollary} \label{cor: meta of completment by deg < 1/2 V}
    Let $\Delta(G)$ denote the maximum degree of $G=(V,E)$.
    If 
    \[ 2\, \Delta(G) < |V|,\] 
    then there exists a graph $G'$ such that $\M(G')=G$.
\end{corollary}

\begin{proof}
    Suppose $xy \in E$. Since $\deg(x) + \deg(y)  < |V|$, there is some $z \in V$ with $xz,yz \notin E$. Theorem \ref{thm: inverse image of M} finishes the result.
\end{proof}

\begin{remark}
    If $G$ is $j$-regular, then Corollary \ref{cor: meta of completment by deg < 1/2 V} says that a sufficient condition for the existence of a $G'$ such that $\M(G')=G$ is $2j < |V|$. Since $|V| = \frac{2|E|}{j}$ in this case, the condition can be rewritten as 
    \[ j^2 < |E|.\]
\end{remark}

\begin{example}
    Theorem \ref{thm: inverse image of M} and Corollary \ref{cor: meta of completment by deg < 1/2 V} allow us to quickly deduce that many families of graphs have the property that each graph is the metamour graph of its complement:
    \begin{itemize}
        \item generalized Petersen graphs $P(n,k)$ for all $n\ge 4$,
        \item cycle graphs $C_n$ for all $n\ge 5$,
        \item path graphs $P_n$ for all $n\ge 5$,
        \item disconnected graphs,
        \item trees $T$ with $\diam(T)\ge 4$, and
        \item undirected Cayley graphs $\Gamma(G,S)$ with $2|S| < |G|$.
    \end{itemize}
    On the other hand, graphs of the following types are not the metamour graphs of any graphs:
    \begin{itemize}
        \item complete $k$-partite graphs for all $k\ge 2$,
        \item complements of generalized Petersen graphs $P(n,k)$ for $n\ge 4$,
        \item complements of cycle graphs $C_n$  for all $n\ge 6$, and
        \item complements of path graphs $P_n$ for all $n\ge 4$.
    \end{itemize}
\end{example}

\section{Graphs with Pseudo-Metamour Period $1$} \label{pseud1}

In this section, we will discuss graphs with metamour period $1$ and graphs with pseudo-metamour period $1$.
We start with the observation that there are only trivial graphs with metamour period $1$.

\begin{theorem}
    A graph $G$ has metamour period $1$ if and only if $G$ is edgeless.
\end{theorem}

\begin{proof}
    With Lemma \ref{lem: subset of G bar}, $M(G)=G$ implies $G = M(G) \subseteq \overline{G}$, and $G=\overline{K}_n$ for some $n\in \Z^+$ follows. The converse is trivial.
\end{proof}

In contrast, there exist many nontrivial examples of graphs with pseudo-metamour period $1$. In the following, we will provide two different families of graphs $G'$ with $M(G')=\overline{G'}\cong G'$, i.e., of metamour-complementary self-complementary graphs $G'$. In both cases we are going to see that any graph $G$ can be embedded as an induced subgraph into some metamour-complementary self-complementary graph. Thus, the class of metamour-complementary self-complementary graphs is large and of a complex structure.

Our first example is based on properties of Paley graphs.

\begin{definition} \label{def: paley}
    Let $q$ be a prime power such that $q\equiv 1 \mod{4}$. Then the \emph{Paley graph} $QR(q)$ has as vertex set the elements of the finite field $\mathbb{F}_q$, with two vertices being adjacent if and only if their difference is a nonzero square in $\mathbb{F}_q$.
\end{definition}

\begin{theorem} \label{thm: paley}
    Every Paley graph is metamour-complementary and self-complementary. Moreover, for any finite graph $G$ there exists a prime power $q\equiv 1 \mod{4}$ such that $G$ embeds as an induced subgraph into $QR(q)$.
\end{theorem}

\begin{proof}
    Let $q$ be a prime power such that $q\equiv 1 \mod{4}$. Then $QR(q)$ is self-complementary \cite{Sachs}, and such that every pair of distinct nonadjacent vertices shares $\frac{q-1}4$ common neighbors \cite[Section~10.3]{GodsilRoyle}. In particular, $\diam(G)=2$, and $G$ is metamour-complementary with Theorem~\ref{M(G) = G^c}. The embedding property follows as Paley graphs are quasi-random   \cite{ChungGrahamWilson}.
\end{proof}

An easier and more instructive example of metamour-complementary self-complementary graphs can be given with the help of the following general graph operation. Some similar constructions will be used in the next section.

\begin{definition} \label{def: join of graphs along a graph}
    Let $G = (V,E)$ be a graph, and let $\mathcal{G}=\{G_v \mid v\in V\}$ be a collection of graphs indexed by $V$. We define the \emph{join of $\mathcal{G}$ along $G$} to be the graph constructed as follows. Begin with $\bigcup_{v\in V} G_v$. For every $vw\in E$, include all possible edges between $G_v$ and $G_w$.
    Thus, for each $vw\in E$, the join $G_v \nabla G_w$ is a subgraph of the join of $\mathcal{G}$ along~$G$.
\end{definition}

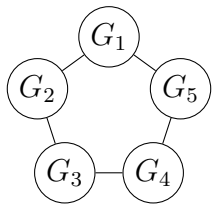
\begin{figure}[h!]
            \begin{center}
            \begin{tikzpicture}[>=stealth, node distance=2cm, every node/.style={circle, draw, minimum size=.5cm,inner sep=2pt},scale=.5]
        
              \node (A) at (90:2cm) {$G_1$};
              \node (B) at (162:2cm) {$G_2$};
              \node (C) at (234:2cm) {$G_3$};
              \node (D) at (306:2cm) {$G_4$};
              \node (E) at (18:2cm) {$G_5$};
            
              \draw[-] (A) -- (B);
              \draw[-] (B) -- (C);
              \draw[-] (C) -- (D);
              \draw[-] (D) -- (E);
              \draw[-] (E) -- (A);
        
            \end{tikzpicture}
            \end{center}
            \caption{The join of $\{G_1, \ldots, G_5\}$ along $C_5$}
            \label{fig: join along C5}
        \end{figure}  

See Figure~\ref{fig: join along C5} for the visualization of Definition~\ref{def: join of graphs along a graph} in the case where $G = C_5$.
(This particular case was constructed by Zelinka~\cite{Zelinka}, just before Proposition~2.)
In the figure, an edge between $G_i$ and $G_j$ represents all of the edges in $G_i \nabla G_j$. 

\begin{theorem} \label{thm: selfcomp}
    Every graph $G$ embeds as an induced subgraph into a metamour-complementary self-complementary graph $G'$ with $|V(G')| = 4\, |V(G)|+1$.
\end{theorem}

\begin{proof}
   Let $G'$ denote the join of the family of graphs $\{G_1, \ldots, G_5\}$ along $C_5$; see Figure~\ref{fig: join along C5}. It is easy to verify that $G'$ is metamour-complementary with $M(G')=\overline{G'}$ as shown in Figure~\ref{fig: metamour along C5}. Note that this graph $M(G')=\overline{G'}$ is (isomorphic to) the join of the family of graphs $\{\overline{G}_1, \overline{G}_3, \overline{G}_5, \overline{G}_2, \overline{G}_4\}$  along $C_5$.
   
\begin{figure}[h!]
            \begin{center}
            \begin{tikzpicture}[>=stealth, node distance=2cm, every node/.style={circle, draw, minimum size=.5cm,inner sep=2pt},scale=.5]
        
              \node (A) at (90:2cm) {$\overline{G}_1$};
              \node (B) at (162:2cm) {$\overline{G}_2$};
              \node (C) at (234:2cm) {$\overline{G}_3$};
              \node (D) at (306:2cm) {$\overline{G}_4$};
              \node (E) at (18:2cm) {$\overline{G}_5$};
            
              \draw[-] (A) -- (C);
              \draw[-] (B) -- (D);
              \draw[-] (C) -- (E);
              \draw[-] (D) -- (A);
              \draw[-] (E) -- (B);
        
            \end{tikzpicture}
            \end{center}
            \caption{$M(G')=\overline{G'}$ for the graph $G$ in Figure~\ref{fig: join along C5}}
            \label{fig: metamour along C5}
        \end{figure}
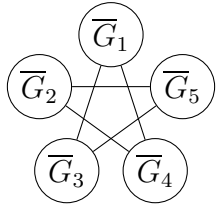  

   In case of the specific choice $G_1 := K_1$ trivial, $G_2 :=G_5 :=G$, and $G_3 := G_4 := \overline{G}$, any choice of isomorphisms $\varphi_1: G_1 \to \overline{G}_1$, $\varphi_2: G_2 \to \overline{G}_3$, $\varphi_3: G_3 \to \overline{G}_5$, $\varphi_4: G_4 \to \overline{G}_2$, and $\varphi_5: G_5 \to \overline{G}_4$ extends to an    isomorphism $\varphi: G' \to \overline{G'}$.\qedhere

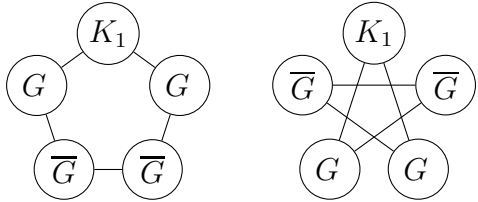
\begin{figure}[h!]
            \begin{center}
            \begin{tikzpicture}[>=stealth, node distance=2cm, every node/.style={circle, draw, minimum size=.5cm,inner sep=2pt},scale=.5]
        
              \node (A) at (90:2cm) {$K_1$};
              \node (B) at (162:2cm) {$\, G\ $};
              \node (C) at (234:2cm) {$\, \overline{G}\, $};
              \node (D) at (306:2cm) {$\, \overline{G}\, $};
              \node (E) at (18:2cm) {$\, G\ $};
            
              \draw[-] (A) -- (B);
              \draw[-] (B) -- (C);
              \draw[-] (C) -- (D);
              \draw[-] (D) -- (E);
              \draw[-] (E) -- (A);
        
            \end{tikzpicture}\qquad
            \begin{tikzpicture}[>=stealth, node distance=2cm, every node/.style={circle, draw, minimum size=.5cm,inner sep=2pt},scale=.5]
        
              \node (A) at (90:2cm) {$K_1$};
              \node (B) at (162:2cm) {$\, \overline{G}\, $};
              \node (C) at (234:2cm) {$\, G\ $};
              \node (D) at (306:2cm) {$\, G\ $};
              \node (E) at (18:2cm) {$\, \overline{G}\, $};
            
              \draw[-] (A) -- (C);
              \draw[-] (B) -- (D);
              \draw[-] (C) -- (E);
              \draw[-] (D) -- (A);
              \draw[-] (E) -- (B);
        
            \end{tikzpicture}
            \end{center}
            \caption{$G'$ and $\overline{G'}$ as in the proof of Theorem \ref{thm: selfcomp}}
            \label{fig: selfcomp}
        \end{figure}
   
\end{proof}

Note that any nontrivial metamour-complementary self-comple\-men\-tary graph has metamour period $2$. With Theorem \ref{thm: paley}, this includes Paley graphs.
We end this section with a general characterization of metamour-complementary graphs with metamour period $2$.

\begin{theorem} \label{thm: metcomplchar}
    A graph $G$ is metamour-complementary with metamour period $2$ if and only if $\diam(G) = \diam(\overline{G}) = 2$.
\end{theorem}

\begin{proof}

    This follows immediately from Theorem \ref{M(G) = G^c}.
    (See also~\cite[Prop.~1]{Zelinka}.)
\end{proof}

We will have more to say about graphs with metamour period $2$ in Section \ref{sec: per2}.

\section{Graphs with Even Metamour Period} 

In this section, we will see that for all even $k\ge 2$, the class of graphs with metamour period $k$ is large and of a complex structure. This is again evidenced by showing that every graph $G$ can be embedded as an induced subgraph into some graph with metamour period $k$ (see Theorem~\ref{thm:induced subgraph}). We will also show that the same is true for graphs with pseudo-metamour period $k$.

We start by introducing an auxiliary sequence of integers.

\begin{definition}
    Let $n$ be an odd positive integer.
    We define 
    \[ \mu(n) \coloneqq \min\{k\in\Z^+ \mid 2^k\equiv\pm1\mod{n}\}. \]
    For example, starting with $n=3$, the first few values of $\mu(n)$ are 
\begin{align*}
1, 2, 3, 3, 5, 6, 4, 4, 9, 6, 11, 10, 9, 14, 5, 5, 12, 18, 12, 10, 7, 12, 23, 21,\\
8, 26, 20, 9, 29, 30, 6, 6, 33, 22, 35, 9, 20, 30, 39, 27, 41, 8, 28, 11, \ldots\
\end{align*}
This sequence can be found as entry A003558 in the OEIS \cite{OEIS}.
\end{definition}


The following result on odd cycle graphs serves as a good comparison of metamour period versus pseudo-metamour period (see Definition~\ref{def:intro} above). Note that we exclude the case $n=3$ below, due to the fact that $C_3 = K_3$ and $\M(K_3)=\overline{K}_3$.
\begin{theorem} \label{thm: odd C_n}
    Let $n \geq 5$ be odd, and let $k,\ell\in\N$.
    We have the following:
    \begin{enumerate}
        \item $\M^k(C_n)\cong C_n$.
        \item $\M^k(C_n)=C_n$ if and only if $\mu(n) \mid k$.
        \item $\M^k(C_n)=\M^\ell(C_n)$ if and only if $k\equiv\ell\mod{\mu(n)}$.
    \end{enumerate}
\end{theorem}

\begin{proof}
    Write the vertices of $C_n$ as $v_i$, $i\in\Z/n\Z$, with edges $v_i v_{i+1}$. Then it is straightforward to verify that the edges of $\M^k(C_n)$ are of the form $v_i v_{i+2^k}$. The statements of the theorem follow.
\end{proof}

\begin{remark}
    Theorem \ref{thm: odd C_n} takes also care of even cycle graphs $C_n$. In particular, given any $n\in \Z^+$, we can write $n=2^i u$ with $i\in \N$ and some odd $u\in \Z^+$. Then $\M^i(C_n)$ is the disjoint union of $2^i$ cycle graphs $C_u$, where we set $C_1 := K_1$.
\end{remark}


In the following, we describe a few more constructions based on our Definition \ref{def: join of graphs along a graph}. We start with a characterization of isomorphic graphs.

\begin{theorem} \label{thm: isomorphic join}
    Let $n \geq 7$ be odd, and let $\mathcal{G}=\{G_i \mid 1\leq i \leq n\}$ and $\mathcal{G}'=\{G'_i \mid 1\leq i \leq n\}$ be two collections of graphs, where we interpret all indices as elements of $\Z/n\Z$. Let $G$ denote the join of $\mathcal{G}$ along $C_n$ and $G'$ the join of $\mathcal{G}'$ along $C_n$, respectively.
    Then $G\cong G'$ if and only if there exists some $j\in \Z/n\Z$ such that either $G'_i\cong G_{j+i}$ or $G'_i\cong G_{j-i}$ for all $1\leq i \leq n$.
\end{theorem}

\begin{proof}
    Note that $\mu(n) \ge 3$ as $n\ge 7$. It is straightforward to verify that $\M^2(G)$ is the join of $\mathcal{G}$ along $\M^2(C_n)$. Thus $G \cap \M^2(G) =\bigcup_{1\leq i \leq n} G_i$, and similarly $\M(G) \cap \M^3(G) =\bigcup_{1\leq i \leq n} \overline{G}_i$. In particular,
    $$E\big(G \cap \M^2(G)\big)\cup E\big(\M(G) \cap \M^3(G)\big)= \bigcup_{1\leq i \leq n} \{xy \mid x,y \in V(G_i), x\ne y\},$$
    and we recover $V(G_i)$ (up to the index). With this information, we can rediscover $\{G_i \mid 1\leq i \leq n\}$ from $G$ up to an index shift/flip.
\end{proof}

We now can generalize Theorem \ref{thm: odd C_n} as follows:

\begin{theorem} \label{thm: odd C_n with decorations}
    Let $n \geq 5$ be odd, and let $\mathcal{G}=\{G_i \mid 1\leq i \leq n\}$ be a collection of graphs with at least one $G_i$ nontrivial. Let $G$ be the join of $\mathcal{G}$ along $C_n$.
    Similarly, let $\overline{\mathcal{G}}=\{\overline{G}_i \mid 1\leq i \leq n\}$ and let $G^0$ be the join of $\overline{\mathcal{G}}$ along $C_n$.  Then we have the following:  
    \begin{enumerate}
        \item $\M^k(G)=G$ if and only if $\mu(n) \mid k$ and $k$ is even.
        \item $\M^k(G)=G^0$ if and only if $\mu(n) \mid k$ and $k$ is odd.
        \item If $n\ge 7$ and $\overline{G}_i\ncong G_i = G_j$ for all $i,j$, then $\M^k(G) \cong \M^\ell(G)$ if and only if $k\equiv \ell \mod 2$. 
    \end{enumerate}
\end{theorem}

\begin{proof}
    It is straightforward to verify that $\M^k(G)$ is the join of either $\mathcal{G}$ (when $k$ is even), or $\overline{\mathcal{G}}$ (when $k$ is odd), along $\M^k(C_n)$. The first two statements of the theorem follow from this and Theorem \ref{thm: odd C_n}. The last statement follows from Theorem \ref{thm: isomorphic join}.
\end{proof}

As an example of Theorem \ref{thm: odd C_n with decorations}, 
let $\mathcal{G}=\{G_i  | \, 1\leq i \leq 7\}$ where every $G_i = \overline{K}_2$, and let $G$ be the join of $\mathcal{G}$ along $C_7$, see Figure \ref{fig:Jonathan's "Dream Catcher" Graph}.
Similarly, let $G^0$ be the join of $\{\overline{G}_i = K_2 \mid  1\leq i \leq 7\}$ along $C_7$.
Since $\mu(7)=3$, we have the following three results from Theorem~\ref{thm: odd C_n with decorations}:
\begin{enumerate}
    \item $\M^k(G)=G$ if and only if $6|k$.
    \item $\M^k(G)=G^0$ if and only if $k$ is an odd multiple of $3$.
    \item $\M^k(G)\cong \M^\ell(G)$ if and only if $k\equiv\ell\mod 2$.
\end{enumerate}

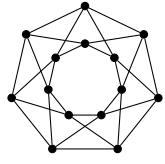
\begin{figure}[h!]
        \centering
            \begin{tikzpicture}[dot/.style={draw,fill,circle,inner sep=1pt},scale=.5]
                \foreach \i in {1,...,7} \node[dot] (w\i) at ( \i*360/7+90:1)   {};
                \foreach \i in {1,...,7} \node[dot] (q\i) at ( \i*360/7+90:2)   {};
                \foreach \i/\j in {1/2,2/3,3/4,4/5,5/6,6/7,7/1} \draw[black] (w\i) -- (w\j);
                \foreach \i/\j in {1/2,2/3,3/4,4/5,5/6,6/7,7/1} \draw[black] (q\i) -- (q\j);
                \foreach \i/\j in {1/2,3/4,5/6,7/1,2/3,4/5,6/7} \draw[black] (w\i) -- (q\j);
                \foreach \i/\j in {1/2,3/4,5/6,7/1,2/3,4/5,6/7} \draw[black] (q\i) -- (w\j);
            \end{tikzpicture}
        \caption{The join of seven copies of $\overline{K}_2$ along $C_7$}
        \label{fig:Jonathan's "Dream Catcher" Graph}
    \end{figure}

We continue with a construction of graphs with even metamour period $k$.

\begin{theorem}
\label{thm:induced subgraph}
    Let $G$ be a graph and $k\in\Z^+$ even. Then $G$ embeds as an induced subgraph into a graph $G'$ with metamour period $k$. In addition, we can achieve $|V(G')|=|V(G)|+4$ for $k=2$ and $|V(G')|=|V(G)|+2^k-2$ for $k\ge 4$.
\end{theorem}

\begin{proof}
    Let $n=5$ for $k=2$ and $n=2^k-1$ for $k\ge 4$. Observe that $n\ge 5$ with $\mu(n)=k$.
    Let $\mathcal{G}=\{G_i \mid 1\leq i \leq n\}$ with $G_1=G$ and $G_i=K_1$ for $2\leq i \leq n$.
    Let $G'$ be the join of $\mathcal{G}$ along $C_n$ so that $G$ embeds as an induced subgraph into $G'$. By Theorem \ref{thm: odd C_n with decorations} (or Theorem~\ref{thm: odd C_n} if $G$ is trivial), $G'$ has metamour period $k$.
\end{proof}

In particular, for any even $k$, there are infinitely many connected graphs whose metamour period is $k$. This result is in stark contrast to our upcoming Conjecture~\ref{conj: odd periodic} (concerning odd metamour periods).

We close this section with a variation of Theorem \ref{thm:induced subgraph} which provides a construction of graphs with even pseudo-metamour period $k$.

\begin{theorem}
\label{thm:induced subgraph2}
    Let $G$ be a graph and $k\in\Z^+$ even. Then $G$ embeds as an induced subgraph into a graph $G'$ with metamour period $k$ and $G'\ncong \M^i(G')$ for all $1\le i< k$. In particular, $G'$ has pseudo-metamour period $k$.
\end{theorem}

\begin{proof}
    Let $n=2^k+1$. Then $n\ge 5$ with $\mu(n)=k$.
    Choose any collection of graphs $\mathcal{G}=\{G_i \mid 1\leq i \leq n\}$ such that $G_1=G$ and all $|V(G_i)|$ are distinct.
    Let $G'$ be the join of $\mathcal{G}$ along $C_n$, and use Theorem~\ref{thm: isomorphic join} for the result.
\end{proof}



\section{More on Graphs With Metamour Period $2$}  \label{sec: per2}

In this section and the next, we study graphs with metamour period $2$ and $3$, respectively. In particular, in this section, we will be focusing on the question whether there exist any graphs
with metamour period $2$ that do not result from joining graphs along $C_5$. We will aim towards a
more complete characterization of graphs with metamour period $2$.

\begin{example}
    It is straightforward to verify that $\M^2(C_5)=C_5$. More generally, by adjoining vertices to $C_5$ as depicted in Figure \ref{C5 + C3}, we can construct additional examples of graphs with metamour period $2$.
    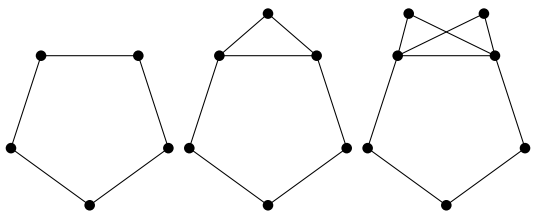
\begin{figure}[h!]
    \begin{center}
        \begin{tikzpicture}
            \foreach \i in {1,...,5}
            {
                \coordinate (P\i) at ({72 * \i - 18}:1.1cm);
            }
            
            \draw (P1) -- (P2) -- (P3) -- (P4) -- (P5) -- cycle;
            
            \foreach \i in {1,...,5}
            {
                \fill (P\i) circle (2pt);
            }
        \end{tikzpicture}
        \begin{tikzpicture}
            \foreach \i in {1,...,5}
            {
                \coordinate (P\i) at ({72 * \i - 18}:1.1cm);
            }
            \coordinate (P6) at (0,1.45);
            
            \draw (P1) -- (P2) -- (P3) -- (P4) -- (P5) -- cycle;
            \draw (P1) -- (P6) -- (P2);
            
            \foreach \i in {1,...,6}
            {
                \fill (P\i) circle (2pt);
            }
        \end{tikzpicture}
        \begin{tikzpicture}
            \foreach \i in {1,...,5}
            {
                \coordinate (P\i) at ({72 * \i - 18}:1.1cm);
            }
            \coordinate (P6) at (-0.5,1.45);
            \coordinate (P7) at (0.5,1.45);
            
            \draw (P1) -- (P2) -- (P3) -- (P4) -- (P5) -- cycle;
            \draw (P1) -- (P6) -- (P2);
            \draw (P1) -- (P7) -- (P2);
            
            \foreach \i in {1,...,7}
            {
                \fill (P\i) circle (2pt);
            }
        \end{tikzpicture}
    \end{center}
    \caption{Some examples of graphs with metamour period~$2$.  The graph in the middle is $\widehat{C}_5$ from Definition~\ref{def:Chat}.}
    \label{C5 + C3}
    \end{figure}
\end{example}

These examples lead to a general construction of an infinite family of graphs with metamour period $2$, constructed via Definition \ref{def: join of graphs along a graph}.
An important element of this construction will be the graph defined below: 

\begin{definition}
\label{def:Chat}
    We write $\widehat{C}_5$ to denote the graph obtained from $C_5$ by adding an additional vertex that is connected to two adjacent vertices of $C_5$.  (The notation is meant to suggest $C_5$ with a ``hat.''
    This is the graph in the center of Figure \ref{C5 + C3}.)
\end{definition}

The graph $\widehat{C}_5$ already provides a simple example of a graph with metamour period $2$ that is not the result of joining graphs along $C_5$. Can we tell any more about the structure of graphs with metamour period $2$? We start with a very general observation.

\begin{theorem} \label{thm:genconstr}
    Let $G$ be a graph with metamour period $2$, and let $G'$ denote the join of a collection of graphs $\mathcal{G}$ along $G$. Then $G'$ has again metamour period $2$.
    In particular, this holds for any join of graphs $\{G_1, \ldots, G_5\}$ along $C_5$, see Figure \ref{fig: join along C5}, and any join of graphs $\{G_1, \ldots, G_6\}$ along $\widehat{C}_5$, see Figure \ref{C_5 party hat with arbitrary graphs}.
    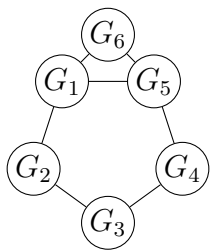
\begin{figure}[h!]
        \centering
        \begin{tikzpicture}[>=stealth, node distance=2cm, every node/.style={circle, draw, minimum size=0cm,inner sep=1pt},scale=.65]
          \node (A) at (-234:1.6cm) {$G_1$};
          \node (B) at (-306:1.6cm) {$G_5$};
          \node (C) at (-18:1.6cm) {$G_4$};
          \node (D) at (-90:1.6cm) {$G_3$};
          \node (E) at (-162:1.6cm) {$G_2$};
          \node (F) at (0,2.25) {$G_6$};
          \draw[-] (A) -- (B);
          \draw[-] (B) -- (C);
          \draw[-] (C) -- (D);
          \draw[-] (D) -- (E);
          \draw[-] (E) -- (A);
          \draw[-] (A) -- (F);
          \draw[-] (B) -- (F);
    \end{tikzpicture}
    \caption{Join of $\{G_1, \ldots, G_6\}$ along $\widehat{C}_5$}
    \label{C_5 party hat with arbitrary graphs}
    \end{figure}
\end{theorem}

\begin{proof}
    Let $\mathcal{G}= \{ G_i \mid i\in I \}$, and denote $\overline{\mathcal{G}}= \{ \overline{G}_i \mid i\in I \}$. Using Definition \ref{def: M}, it is straightforward to check that $\M(G')$ 
    is the join of $\overline{\mathcal{G}}$ along $\M(G)$.    
    From this, it is straightforward to verify that \mbox{$\M^2(G')=G'$}.
\end{proof}

\begin{example}
   The graph on the right side of Figure \ref{C5 + C3} is the join of $\{K_1,K_1,K_1,K_1,K_1,\overline{K}_2\}$ along $\widehat{C}_5$.

   Note also that every nontrivial metamour-complementary self-com\-ple\-mentary graph $G$ has metamour period $2$ and thus qualifies for Theorem \ref{thm:genconstr}. This includes all Paley graphs, see Theorem \ref{thm: paley}. 
\end{example}
    

Our next result can be viewed as a generalization of Theorem \ref{thm: metcomplchar}.

\begin{theorem} \label{thm: period 2 metamour diameter constraints}
    Suppose $G=(V,E)$ is a connected graph with metamour period $2$. Then $\M(G)$ is connected and either
    \[ \diam(G) = \diam(\M(G)) = 2 \quad \textup{with} \quad \M(G)=\overline{G},\]
    or
    \[ \diam(G) = \diam(\M(G)) = 3 \quad \textup{with} \quad \M(G)\subsetneq\overline{G}.\]
    Moreover, for all distinct $x,y\in V$, we have
    \begin{align*} 
        d_G(x,y)=1 &\iff d_{\M(G)}(x,y) = 2,\\ 
        d_G(x,y)=2 &\iff d_{\M(G)}(x,y) = 1,\\ 
        d_G(x,y)=3 &\iff d_{\M(G)}(x,y) = 3.
    \end{align*}
    
\end{theorem}

\begin{proof}
    Since $\M^2(G) = G$ is connected, $\M(G)$ is connected. It follows that $\diam(G) \geq 2$ since $\M(K_n)=\overline{K}_n$ is disconnected for $n\ge 2$ while $K_1$ has metamour period $1$.
    Moreover, if there existed $x,y\in V$ with $d_G(x,y)=4$, then there would exist $z\in V$ so that $d_G(x,z)=d_G(z,y)=2$. This would force $xz,yz\in E(\M(G))$ and $xy\not\in E(\M(G))$. By metamour periodicity, this would show that $xy\in E$, a contradiction. As $\M(G)$ also has metamour period $2$, the roles of $G$ and $\M(G)$ are symmetrical and we see that
    \[ 2 \leq \diam(G), \diam(\M(G)) \leq 3.\]

    Now, if $\diam(G)=2$, then for distinct $x,y\in V$, either $xy\in E$ or $xy \in \M(G)$. From this, it follows that $\M(G) = \overline{G}$. By periodicity and Theorem \ref{thm: inverse image of M}, $\diam(\overline{G})=2$ as well. By symmetry, we see that 
    \[ \diam(G) =2 \iff \diam(\M(G)) = 2. \]

    From this, we also conclude that $\diam(G) =3 \iff \diam(\M(G)) = 3$. In this case, the existence of distance-three vertices implies that $\M(G)\neq\overline{G}$.

    Now let $x,y\in V$ be distinct. By periodicity, $d_G(x,y)=1$ implies $d_{\M(G)}(x,y) = 2$.
    Moreover, by Definition \ref{def: M}, $d_G(x,y)=2$ implies $d_{\M(G)}(x,y) = 1$. 
    By symmetry, we thus obtain both equivalences involving distance 1 and 2.
    In turn, this yields the equivalence involving distance $3$. 
\end{proof} 



\begin{theorem} \label{thm: C5 with party hat}
    Suppose $G$ is a connected graph with metamour period $2$ and $\diam(G) = 3$. Then $G$ contains a subgraph isomorphic to $\widehat{C}_5$ $($shown in the middle of Figure~\ref{C5 + C3}$)$.
\end{theorem}

\begin{proof}
    Let $x,y \in V(G)$ with $d_G(x,y) = 3$. As $d_{\M(G)}(x,y) = 3$, let $(w_0,w_1,w_2,w_3)$ be a minimal path from $x$ to $y$ in $\M(G)$. For $0\leq i \leq 2$, since $w_i w_{i+1} \in E(\M(G))$, there is some $u'_i \in V(G)$ with $w_iu'_i,u'_iw_{i+1} \in E(G)$. 
    
    Define a walk $(u_0,u_1,\dots,u_6)$ by
    \[
        u_i =
        \begin{cases}
            w_{i/2}, & \text{for $i$ even}, \\
            u'_{(i-1)/2}, &\text{otherwise.}
        \end{cases}
    \]
    After possibly relabeling (see Figure \ref{Path of length 3}), it is straightforward to verify that there are two possibilities:
    \begin{enumerate}
        \item All of the vertices $u_i$ are distinct.
        \item The only vertices $u_i$ that are not distinct are $u_3 = u_5$.
    \end{enumerate}
    
    By minimality, we have $u_0u_4,u_2u_6\in E(G)$, giving us the two possible configurations depicted in Figure \ref{Path of length 3}.    
        \begin{figure}[h!]
        \begin{center}
            \begin{tikzpicture}
              \coordinate (u0) at (-1,1);
              \coordinate (u1) at (-2,0);
              \coordinate (u2) at (-1,-1);
              \coordinate (u3) at (0,-2);
              \coordinate (u4) at (1,-1);
              \coordinate (u5) at (2,0);
              \coordinate (u6) at (1,1);
            
              \draw [blue,thick,dotted] (u0) -- (u2) -- (u4) -- (u6);
              \draw (u0) -- (u1) -- (u2) -- (u3) -- (u4) -- (u5) -- (u6);
    
              \foreach \i in {0,...,6}
                {
                    \fill (u\i) circle (2pt);
                }
            
              \foreach \i in {0,...,3}
                {
                    \node[anchor=east] at (u\i) {$u_{\i}$};
                }
            \foreach \i in {4,...,6}
                {
                    \node[anchor=west] at (u\i) {$u_{\i}$};
                }
    
                \draw (u0) -- (u4);
                \draw (u2) -- (u6);
            \end{tikzpicture}
            \begin{tikzpicture}
              \coordinate (u0) at (-1,1);
              \coordinate (u1) at (-2,0);
              \coordinate (u2) at (-1,-1);
              \coordinate (u3) at (0,-2);
              \coordinate (u4) at (1,-1);
              \coordinate (u5) at (1,1);
            
              \draw [blue,thick,dotted] (u0) -- (u2) -- (u4) -- (u6);
              \draw (u0) -- (u1) -- (u2) -- (u3) -- (u4);
    
              \foreach \i in {0,...,5}
                {
                    \fill (u\i) circle (2pt);
                }
            
              \foreach \i in {0,...,2}
                {
                    \node[anchor=east] at (u\i) {$u_{\i}$};
                }
            \node[anchor=north] at (u3) {$u_3=u_5$};
            \node[anchor=west] at (u4) {$u_4$};
            \node[anchor=west] at (u5) {$u_6$};
    
                \draw (u0) -- (u4);
                \draw (u3) -- (u5);
                \draw (u2) -- (u5);
            \end{tikzpicture}
            \end{center}
        \caption{Two possible paths of length $3$ in $\M(G)$}
        \label{Path of length 3}
        \end{figure}
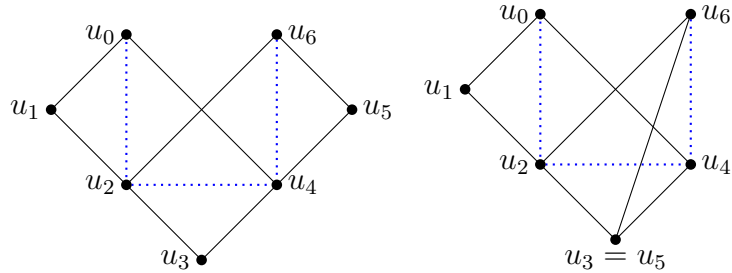        
        If the vertices are all distinct, then, after possibly relabeling, it is straightforward to check that $u_0 u_3 \in E(\M(G))$ and $u_3 u_6 \in E(G)$. As a result, in either possibility from above, we have a copy of $\widehat{C}_5$  (see Figure \ref{C_5 party hat subgraph}).
        \begin{figure}[h!]
        \begin{center}
            \begin{tikzpicture}
              \foreach \i in {1,...,5}
              {
                \coordinate (P\i) at ({72 * \i - 18}:1.25cm);
              }
              \coordinate (P6) at (0,1.75);
            
              \draw (P1) -- (P2) -- (P3) -- (P4) -- (P5) -- cycle;
              \draw (P1) -- (P6) -- (P2);
            
              \foreach \i in {1,...,6}
              {
                \fill (P\i) circle (2pt);
              }
              \node[anchor=west] at (P1) {$u_2$};
              \node[anchor=east] at (P2) {$u_3$};
              \node[anchor=east] at (P3) {$u_4$};
              \node[anchor=west] at (P4) {$u_0$};
              \node[anchor=west] at (P5) {$u_1$};
              \node[anchor=west] at (P6) {$u_6$};
              
            \end{tikzpicture}
        \end{center}
    \caption{$\widehat{C}_5$ with labeled vertices}
    \label{C_5 party hat subgraph}
    \end{figure}
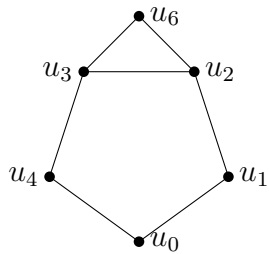
\end{proof}

In light of Theorem \ref{thm: C5 with party hat}, it seems reasonable to ask for a full characterization of diameter-$3$ graphs with metamour period $2$ (see Theorem~\ref{thm: period 2 metamour diameter constraints}). In particular, we have the following question:
\begin{question} \label{q1}
    Is every connected diameter-$3$ graph with metamour period $2$ the join of some graphs along $\widehat{C}_5$ $($as in Figure~\ref{C_5 party hat with arbitrary graphs}$)$?
\end{question}

A negative answer to Question \ref{q1} may result in the discovery of a whole new family of diameter-$3$ graphs with metamour period $2$ similar to Paley graphs as diameter-$2$ graphs with metamour period $2$.

In this context, one may also ask the following question:
\begin{question} \label{q2}
    Is every connected diameter-$2$ graph with metamour period $2$ the join of some graphs along a Paley graph?
\end{question}



\section{Graphs With Metamour Period $3$} 

In contrast to Section \ref{sec: per2}, we will provide a full characterization of graphs with metamour period $3$ in Theorem \ref{thm: classification period 3} and Corollary \ref{cor: classification period 3}. 

En route to these results, we begin with a slate of lemmas.
\begin{lemma} \label{Disjointedness}
    Let $G$ be a connected graph with metamour period $3$.
    \begin{enumerate}
        \item Then $E(G)$, $E(\M(G))$, and $E(\M^2(G))$ are pairwise disjoint.
        \item If $v_1v_2,v_2v_3 \in E(G')$ for some $G' \in \{G,\M(G),\M^2(G)\}$, then $v_1v_3 \notin E(\M^2(G'))$.
        \item If $v_1v_2,v_2v_3 \in E(G')$ and $v_1v_2',v_2'v_3 \in E(\M(G'))$ for some $G' \in \{G,\M(G),\M^2(G)\}$, then $v_1v_3 \in E(\M(G'))$.
    \end{enumerate}
\end{lemma}

\begin{proof}
    The first part follows from Lemma \ref{lem: subset of G bar} applied to each of the pairs of graphs $\M^i(G)$ and $\M^{i+1}(G)$ for $0\leq i \leq 2$ combined with metamour period $3$. 
    For the second part, Definition \ref{def: M} shows that $v_1v_2,v_2v_3 \in E(G')$ implies $v_1 v_3$ is in either $E(G')$ or $E(\M(G'))$. Combined with part one, we are done. 
    For the third part, $v_1v_2',v_2'v_3 \in E(\M(G'))$ implies $v_1 v_3$ is in either $E(\M(G'))$ or $E(\M^2(G'))$. Combined with part two, we are done.
\end{proof}

\begin{lemma} \label{trail condition}
    Suppose $G=(V,E)$ is a connected graph with metamour period $3$. For $v_1v_2 \in E$, there is a walk $(w_0,w_1,\dots,w_8)$ in $G$ such that $w_0=v_1$, $w_8 = v_2$, $w_0w_2, w_2w_4, w_4w_6, w_6w_8 \in E(\M(G))$, and $w_0w_4, w_4w_8 \in E(\M^2(G))$, see Figure \ref{Adjacency in M^3(G)}. Furthermore, one of the following must hold:
    \begin{enumerate}
        \item All vertices $w_i$ are distinct.
        \item The only vertices $w_i$ that are not distinct are $w_0 = w_7$ and $w_1 = w_8$.
    \end{enumerate}
\end{lemma}

\begin{figure}[h!]
    \begin{center}
        \begin{tikzpicture}
          \coordinate (w0) at (-1.175,1.175);
          \coordinate (w1) at (-2.55,0.35);
          \coordinate (w2) at (-2.2,-1);
          \coordinate (w3) at (-0.75,-1.2);
          \coordinate (w4) at (0,0);
          \coordinate (w5) at (0.75,-1.2);
          \coordinate (w6) at (2.2,-1);
          \coordinate (w7) at (2.55,0.35);
          \coordinate (w8) at (1.175,1.175);
          
          \foreach \i in {0,...,8}
          {
            \filldraw (w\i) circle (2pt);
          }
        
          \draw (w0) -- (w8);
          \foreach \i in {0,...,7}
          {
            \pgfmathtruncatemacro{\j}{\i+1}
            \draw (w\i) -- (w\j);
          }
          \foreach \i in {0,2,4,6}
          {
            \pgfmathtruncatemacro{\j}{\i+2}
            \draw [blue,thick,dotted] (w\i) -- (w\j);
          }
          \draw[red,thick,dashed] (w0) -- (w4);
          \draw[red,thick,dashed] (w4) -- (w8);
        
          \node[anchor=south, inner sep=4pt] at (w0) {$w_{0}=v_1$};
          \node[anchor=east, inner sep=4pt] at (w1) {$w_{1}$};
          \node[anchor=north, inner sep=4pt] at (w2) {$w_{2}$};
          \node[anchor=north, inner sep=4pt] at (w3) {$w_{3}$};
          \node[anchor=north, inner sep=10pt] at (w4) {$w_{4}$};
          \node[anchor=north, inner sep=4pt] at (w5) {$w_{5}$};
          \node[anchor=north, inner sep=4pt] at (w6) {$w_{6}$};
          \node[anchor=west, inner sep=4pt] at (w7) {$w_{7}$};
          \node[anchor=south, inner sep=4pt] at (w8) {$w_{8}=v_2$};
        
        \end{tikzpicture}
    \end{center}
\caption{\mbox{Eight walk in a graph with metamour period $3$}}
\label{Adjacency in M^3(G)}
\end{figure}

\begin{proof}
    Begin with $w_0=v_1$ and $w_8=v_2$ so that $w_0 w_8 \in E$. 
    By metamour period $3$, there exists $w_4 \in V$ so that $w_0 w_4, w_4 w_8 \in E(\M^2(G))$. 
    From this it follows that there exist $w_2, w_6 \in V$ so that $w_0w_2$, $w_2w_4$, $w_4w_6$, $w_6w_8 \in E(\M(G))$.
    Finally, this shows that there exist $w_1$, $w_3$, $w_5$, $w_7 \in V$ so that $(w_0,w_1,\dots,w_8)$ in a walk in $G$.
    
    By construction, any pair of adjacent vertices in Figure \ref{Adjacency in M^3(G)} must be distinct in $G$ (solid lines), respectively $\M(G)$ (dotted lines), and $\M^2(G)$ (dashed lines). Furthermore, Lemma \ref{Disjointedness} shows that $w_0 \notin \{w_3,w_5,w_6\}$, $w_1 \notin \{w_3,w_4,w_5,w_6\}$, and $w_2 \notin \{w_5,w_6,w_7,w_8\}$.

    Next we show that $w_3$ and $w_5$ are distinct. By way of contradiction, suppose $w_3 = w_5$. From Lemma \ref{Disjointedness}(3), we would have $w_2w_6 \in E(\M(G))$. Therefore, either $w_0w_6 \in E(\M(G))$, which implies $w_0w_8 \notin E$, or $w_0w_6 \in E(\M^2(G))$, which implies $w_4w_6 \notin E(\M(G))$. Contradiction.
    By similar arguments, $w_1 \neq w_7$ and $w_1w_3,w_5w_7 \in E(\M(G))$.

    By symmetry, it only remains to show that assuming $w_0 = w_7$ and $w_1 \neq w_8$ leads to a contradiction, see Figure \ref{C_7 + 1}.
    \begin{figure}[h!]
    \begin{center}
        \begin{tikzpicture}
              \foreach \i in {1,...,7}
              {
                \coordinate (P\i) at (-{360/7 * (\i-2)}:1.6cm);
              }
              \coordinate (P8) at (0.7,2.3);
            
              \draw (P1) -- (P2) -- (P3) -- (P4) -- (P5) -- (P6) -- (P7) -- cycle;
              \draw (P1) -- (P8);
            
              \foreach \i in {1,...,8}
              {
                \fill (P\i) circle (2pt);
              }
              \node[anchor=west, inner sep=6pt] at (P1) {$w_0 = w_7 = v_1$};
              \node[anchor=west, inner sep=6pt] at (P8) {$w_8 = v_2$};
              \foreach \i in {2,...,7}
              {
                \pgfmathtruncatemacro{\j}{\i-1}
                \node[anchor={-360/7 * (\i-3) - 180}, inner sep=6pt] at (P\i) {$w_{\j}$};
              }
    
              \foreach \i in {1,3,5}
              {
                \pgfmathtruncatemacro{\j}{\i+2}
                \draw [blue,thick,dotted] (P\i) -- (P\j);
              }
              \draw [blue,thick,dotted] (P7) -- (P8);
              \draw [blue,thick,dotted] (P2) -- (P4);
              \draw [blue,thick,dotted] (P6) -- (P1);
              \draw[red,thick,dashed] (P1) -- (P5);
              \draw[red,thick,dashed] (P5) -- (P8);
            \end{tikzpicture}
        \end{center}
    \caption{$w_0 = w_7$ and $w_1 \neq w_8$}
    \label{C_7 + 1}
    \end{figure}
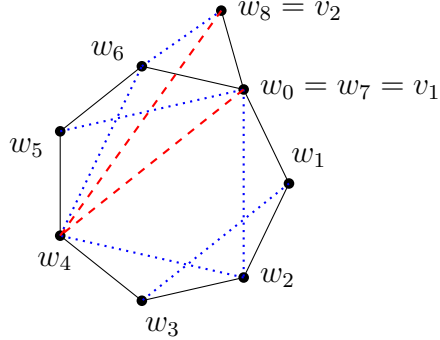

    Case 1: $w_1w_8 \in E$.
    Since $w_0w_8\in E$, we get $w_8w_2 \notin E(\M(G))$ and hence $w_8w_2 \in E$. Thus, either $w_8w_3 \in E$, which implies $w_8w_4 \notin E(\M^2(G))$, or $w_8w_3 \in E(\M(G))$, which implies $w_1w_8 \notin E$. Contradiction.

    Case 2: $w_1w_8 \in E(\M(G))$.
    By Lemma \ref{Disjointedness}, we have $w_1w_6 \in E(\M(G))$. Since $w_1w_6,w_6w_4 \in E(\M(G))$, we get either $w_1w_4 \in E(\M(G))$, which implies $w_1w_2 \notin E$, or $w_1w_4 \in E(\M^2(G))$, which implies $w_8w_1 \notin E(\M(G))$. Contradiction.
\end{proof}

\begin{lemma} \label{Adjacent pairs}
    Suppose $G$ is a connected graph with metamour period $3$. Then $G$ contains an induced copy of either $C_7$ or $C_9$.
\end{lemma}

\begin{proof}
    Begin with the walk from Lemma \ref{trail condition}, $(w_0,\ldots,w_8)$. If $w_0 = w_7$ and $w_1 = w_8$, we will see that we get an induced $C_7$. When all $w_i$ are distinct, we will get an induced $C_9$. As the arguments are similar and overlap, we give details here only for the first case. 
    
    Suppose that $w_0 = w_7$, $w_1 = w_8$, and all other $w_i$ are distinct.    
    By Lemma \ref{Disjointedness}(1), we see that $w_5w_0$,\! $w_0w_2$,\! $w_2w_4$,\! $w_4w_6$,\! $w_6w_1$,\! $w_1w_3$,\! $w_1w_4$,\! $w_4w_0 \notin E(G)$. Since $w_6w_4,w_4w_2 \in E(\M(G))$, we have $w_6w_2 \notin E(G)$ from Lemma \ref{Disjointedness}(2). Similar arguments show $w_2w_5, w_3w_6 \notin E(G)$. Since $w_4w_0 \in E(\M^2(G))$, we have $w_0w_3 \notin E(G)$. Similarly, $w_5w_1 \notin E(G)$.
    
    Finally, suppose $w_3w_5 \in E(G)$. Then $w_2w_5 \in E(\M(G))$. Since $w_4w_2,w_2w_5 \in E(\M(G))$, we get $w_4w_5 \notin E(G)$. Contradiction.
\end{proof}

\begin{lemma} \label{lem: c7 c9 with M2 condition}
    Suppose $G$ is a connected graph with  metamour period~$3$. 
    Then $G$ contains an induced copy of $C_n$, $n\in\{7,9\}$, on vertices $w_i$, $0\leq i \leq n-1$, so that $w_i w_{i\pm 4} \in E(\M^2(G))$ for all $i$, with indices interpreted $\mod n$.
\end{lemma}

\begin{proof}
    Continue the notation and arguments from Lemmas \ref{trail condition} and \ref{Adjacent pairs}.
    We only give details in the case where all vertices $w_i$ are distinct since the case of $w_0 = w_7$ and $w_1 = w_8$ is similar and straightforward.

    Note that $w_i w_{i\pm 2} \in E(\M(G))$ for all $i$. Thus, $w_0w_2,w_2w_4 \in E(\M(G))$, and either $w_0w_4 \in E(\M(G))$ or $w_0w_4 \in E(\M^2(G))$. Suppose $w_0w_4 \in E(\M(G))$. Then either $w_0w_6 \in E(\M(G))$, which implies $w_0w_8 \notin E(G)$, or $w_0w_6 \in E(\M^2(G))$, which together with $w_0w_4 \in E(\M^2(G))$ implies $w_4w_6 \notin E(\M(G))$. Contradiction.
    
    The others statements follow by similar arguments.
\end{proof}

We arrive at the main theorem of this section.

\begin{theorem} \label{thm: classification period 3}
    $G=(V,E)$ is a connected graph with metamour period $3$ if and only if it is isomorphic to either $C_7$ or $C_9$.
\end{theorem}

\begin{proof}
    Let $G$ be a connected graph with metamour period $3$. 
    Begin with the induced copy of $C_n$, $n\in\{7,9\}$, on $w_0,\ldots,w_{n-1}$ from Lemma~\ref{lem: c7 c9 with M2 condition}. It remains to show that $G$ has no additional vertices.
    By way of contradiction, suppose there exists $v\in V \setminus \{w_0,\dots,w_{n-1}\}$.
    After relabeling, suppose $vw_0 \in E$. 
    
    Then either $vw_1 \in E$ or $vw_1 \in E(\M(G))$. By way of contradiction, suppose $vw_1 \in E$. 
    As a result, either $v w_2 \in E$ or $v w_2 \in E(\M(G))$. In the later case, combining with $w_0 w_2 \in E(\M(G))$, we would get $v w_0 \notin E$. Contradiction. Therefore, $v w_2 \in E$.     
    We can similarly conclude $vw_3,vw_4 \in E$. As then $w_0v,vw_4 \in E$, it follows that $w_0 w_4 \notin E(\M^2(G))$. Contradiction. Thus, $vw_1 \in E(\M(G))$. Similarly, $vw_{n-1} \in E(\M(G))$.

    Since $vw_1, w_1w_3 \in E(\M(G))$, we get either $v w_3 \in E(\M(G))$ or $v w_3 \in E(\M^2(G))$.
    If $vw_3 \in E(\M^2(G))$, then $vw_3,w_3w_{n-1} \in E(\M^2(G))$ by Lemma~\ref{lem: c7 c9 with M2 condition}, and so $vw_{n-1} \notin E(\M(G))$. Contradiction. Thus $vw_3 \in E(\M(G))$, and either  $vw_5 \in E(\M(G))$ or $vw_5 \in E(\M^2(G))$. If $vw_5 \in E(\M^2(G))$, then $w_1w_5 \in E(\M^2(G))$ implies the contradiction $vw_1 \notin E(\M(G))$. Thus, $vw_5 \in E(\M(G))$.

    Case 1: $w_0 = w_7$ and $w_1 = w_8$ with $n=7$.
    
    
    With $vw_5, vw_{n-1} = vw_6 \in E(\M(G))$, we have the final contradiction $w_5w_6 \notin E$. 

    Case 2: The vertices are all distinct with $n=9$.
    
    With $vw_5 \in E(\M(G))$ we also have $vw_4 \in E(\M(G))$ by symmetry. Now $vw_4, vw_5 \in E(\M(G))$ implies $w_4w_5 \notin E$. Contradiction.
\end{proof}

The following result is now immediate.

\begin{corollary} \label{cor: classification period 3}
    A nontrivial graph has metamour period $3$ if and only if it is the disjoint union of some copies of $C_7$ and $C_9$.
\end{corollary}

In general, we conjecture that connected graphs with odd metamour period are rare. If true, it would be especially interesting to classify them. Note the conjectured difference to graphs with even metamour period, Theorem \ref{thm:induced subgraph}.

\begin{conjecture} \label{conj: odd periodic}
    For each odd $k\in \Z^+$, there exist only finitely many connected graphs with metamour period $k$.
\end{conjecture}

\section{Metamours of Generalized Petersen Graphs}

To help with digestion of Definition \ref{def:2walk} below, we will begin with a walk of $2^k + 1$ vertices and $2^k$ associated edges. For each $i$, $0 \leq i \leq k-1$, this walk will be broken up into $2^{k-i-1}$ smaller walks of $2^{i+1} + 1$ consecutive vertices and $2^{i+1}$ edges. For the $j$th such smaller walk, $0 \leq j \leq 2^{k-i-1}-1$, we will look at its first, middle, and last vertex. In particular, there will be $2^i$ edges each between the middle vertex and the vertices at either end of this small walk.
\begin{definition} \label{def:2walk}
    Fix $G=(E,V)$, $u,v \in V$ distinct vertices, and $k\in\Z^+$. A \emph{$2$-walk of length $k$} from $u$ to $v$ is a walk $\pi = (w_0,w_1,\dots,w_{2^k})$ in $G$ with $u=w_0$ and $v=w_{2^k}$ such that:    
    \begin{enumerate}
        \item For all $0 \leq i \leq k-1$ and $0 \leq j \leq 2^{k-i-1}-1$, 
            $w_{2j\cdot2^i}$, $w_{(2j+1)\cdot2^i}$, and $w_{2(j+1)\cdot2^i}$ are distinct.
        \item For all $0 \leq j \leq 2^{k-1}-1$, $w_{2j}w_{2(j+1)} \notin E$.
    \end{enumerate}
    Note that restriction of $\pi$ to $(w_{2j\cdot 2^i}, \ldots, w_{2(j+1)\cdot 2^i})$ gives a $2$-walk of length $i+1$ from $w_{j\cdot 2^{i+1}}=w_{2j\cdot 2^i}$ to $w_{(j+1)\cdot 2^{i+1}} =w_{2(j+1)\cdot 2^i}$.
    We will say that $\pi$ is \emph{fully minimal} if, for all $1 \leq i \leq k-1$ and $0 \leq j \leq 2^{k-i-1}-1$, there is no $2$-walk from $w_{j\cdot 2^{i+1}}$ to $w_{(j+1)\cdot 2^{i+1}}$ of length $i$. Again, restricting a fully minimal $2$-walk $\pi$ to $(w_{j\cdot 2^{i+1}}, \ldots, w_{(j+1)\cdot 2^{i+1}})$ gives a fully minimal $2$-walk of length $i+1$ from $w_{j\cdot 2^{i+1}}$ to $w_{(j+1)\cdot 2^{i+1}}$.
\end{definition}

\begin{remark} \label{rmk: min min is fully minimal}
    Continue the notation from Definition \ref{def:2walk}. We will say that 
    \[ d_2(u,v) = i \]
    if the minimal length of a $2$-walk between $u$ and $v$ is $i$. 
    
    With this notation, if a $2$-walk $\pi = (w_0,w_1,\dots,w_{2^k})$ of length $k$ 
    from $u$ to $v$ satisfies $d_2(w_{j\cdot 2^{i+1}}, w_{(j+1)\cdot 2^{i+1}})=i+1$ for all $1 \leq i \leq k-1$ and $0 \leq j \leq 2^{k-i-1}-1$, then $\pi$ will trivially be fully minimal. However, the converse of this statement is not true. Also, neither is it true that every minimal length $2$-walk is fully minimal nor that every fully minimal $2$-walk is of minimal length.
\end{remark}

In the case of $k=1$ in Definition \ref{def:2walk}, the existence of a (fully minimal) $2$-walk of length $k$ from $u$ to $v$ is equivalent to $uv\in E(\M^1(G))$. However, this is no longer true when $k\geq 2$. To that end, from Definition \ref{def: M}, we immediately get the following condition to have an edge in $\M^k(G)$.
\begin{lemma} \label{lemma:edge in Mk implies walk}
    Let $G=(V,E)$, $u,v \in V$ distinct vertices, and $k\in \Z^+$. Then $uv \in E(\M^k(G))$ if and only if there is a $2$-walk $(w_0,w_1,\dots,w_{2^k})$ of length $k$ from $u$ to $v$ in $G$ such that, for $0 \leq i \leq k$ and $0 \leq j \leq 2^{k-i} - 1$, 
    \[ w_{j\cdot 2^i}w_{(j+1)\cdot2^i} \in E(\M^{i}(G)).\]
\end{lemma}

Lemma \ref{lemma:edge in Mk implies walk} gives a necessary and sufficient condition to have an edge in $\M^k(G)$, but it is very hard to verify. Theorem \ref{thm: rel min walk implies edge in Mk} below gives a sufficient condition that is easier to check if it is satisfied.
\begin{theorem} \label{thm: rel min walk implies edge in Mk}
    Let $G=(V,E)$ with distinct vertices $u,v \in V$. If there is a fully minimal $2$-walk of length $k$ from $u$ to $v$ in $G$, then $uv \in E(\M^k(G))$.
\end{theorem}

\begin{proof}
    Let $(w_0,w_1,\dots,w_{2^k})$ be a fully minimal $2$-walk of length $k$ from $u$ to $v$. As $w_{2j}w_{2j+1},w_{2j+1}w_{2(j+1)} \in E$ for $0 \leq j \leq 2^{k-1}-1$ and $w_{2j}w_{2(j+1)} \notin E$, we get $w_{2j}w_{2(j+1)} \in E(\M^1(G))$.

    If $uv\not\in E(\M^k(G))$, then, noting Lemma \ref{lemma:edge in Mk implies walk}, choose the smallest $i$, $1 \leq i \leq k-1$, such that there is some $j$, $0 \leq j \leq 2^{k-i-1} - 1$, 
    so that $w_{j\cdot 2^{i+1}}w_{(j+1)\cdot2^{i+1}} \notin E(\M^{i+1}(G))$. 
 
    However, minimality of $i$ gives
    $w_{2j\cdot 2^i}w_{(2j+1)\cdot2^i}, \, w_{(2j+1)\cdot2^i}w_{2(j+1)\cdot2^i} \in E(\M^i(G))$. Thus, we must have 
    $w_{j\cdot 2^{i+1}}w_{(j+1)\cdot2^{i+1}}=w_{2j\cdot 2^i} w_{2(j+1)\cdot2^i} \in E(\M^i(G))$ as $w_{j\cdot 2^{i+1}}w_{(j+1)\cdot2^{i+1}} \notin E(\M^{i+1}(G))$.   
    Lemma \ref{lemma:edge in Mk implies walk} now shows that there exists a $2$-walk of length $i$ from $w_{j\cdot 2^{i+1}}$ to $w_{(j+1)\cdot2^{i+1}}$, which is a contradiction.
\end{proof}

The next lemma will allow us to bootstrap up metamour orders by expanding $2$-walks.
\begin{lemma}  \label{Expansion of 2-walks}
    Let $uv \in E(\M^k(G))$ and $\pi =( w_0,w_1,\dots,w_{2^k})$ a $2$-walk from $u$ to $v$ in $G$ such that, for $0 \leq i \leq k$ and $0 \leq j \leq 2^{k-i} - 1$, $w_{j\cdot2^i}w_{(j+1)\cdot2^i} \in E(\M^{i}(G))$. 
    
    Suppose that, for $0 \leq a \leq 2^k - 1$, there is a fully minimal $2$-walk of length $2$ from $w_a$ to $w_{a+1}$. Then $uv \in E(\M^{k+2}(G))$ as well.
\end{lemma}

\begin{proof}
    Construct a $2$-walk, $\pi' = (w_0',w_1',\dots,w_{2^{k+2}}')$, of length $k+2$ from $\pi$ by replacing each edge $w_a w_{a+1}$ in $\pi$ by its corresponding fully minimal $2$-walk of length $2$ from $w_a$ to $w_{a+1}$. Observe that $w'_{4a} = w_a$. By construction, note that, for $0 \leq j \leq 2^{k+1}-1$, we have $w'_{2j}w'_{2(j+1)} \in E(\M^1(G))$.
    
    Arguing as in Lemma \ref{thm: rel min walk implies edge in Mk}, suppose $uv \not\in E(\M^{k+2}(G))$. Using Lemma \ref{lemma:edge in Mk implies walk}, choose the smallest $i$, $1 \leq i \leq k+1$, such that there is some $j$, $0 \leq j \leq 2^{k-i-1} - 1$, so that $w_{j\cdot 2^{i+1}}' w_{(j+1)\cdot2^{i+1}}' \notin E(\M^{i+1}(G))$. By full minimality of the added $2$-walks, we see that $i\geq 2$.
    By minimality of $i$, $w'_{2j\cdot2^i}w'_{(2j+1)\cdot2^i},w'_{(2j+1)\cdot2^i}w'_{2(j+1)\cdot2^i} \in E(\M^i(G))$
    which forces $w_{j\cdot 2^{i+1}}' w_{(j+1)\cdot2^{i+1}}' \in E(\M^i(G))$. Thus 
    $w_{j\cdot2^{i-1}} \, w_{(j+1)\cdot2^{i-1}} \in E(\M^i(G))$, which violates its membership in $E(\M^{i-1}(G))$.
\end{proof}

    

Write $G(m,j)$ for the \emph{generalized Petersen graph} where $m,j\in\Z^+$ with $m\geq 5$ and $1\leq j < \frac{m}{2}$. We will use $\{v_i, u_i \mid 0\le i < m\}$ as vertex set with edges
\[ v_i v_{i+1}, \,\, v_i u_i, \text{ and } u_i u_{i+j} \text{ for all } 0\le i<m, \]
where indices are to be read modulo $m$. We may refer to the $\{v_i\}$ as the \emph{exterior vertices} and the $\{u_i\}$ as the \emph{interior vertices.}
Observe that the interior vertices break up into $(m,j)$ cycles of size $\frac{m}{(m,j)}$ each, where $(m,j)$ denotes the greatest common divisor of $m$ and $j$.

Our first main result, Theorem \ref{thm: Peteresen G(m,2) limit sets}, will calculate the metamour limit period and metamour limit set of $G(m,2)$.
With an eye towards applying Lemma \ref{Expansion of 2-walks} in the context of certain generalized Petersen graphs, we prove the following lemma.
\begin{lemma} \label{2-walk of length 2 in G(m,2)}
    Let $m\in\Z^+$ with $m\geq 5$. If $uv \in E(G(m,2))$, there exists a fully minimal $2$-walk of length $2$ from $u$ to $v$.
\end{lemma}

\begin{proof}
    It will be sufficient to show that every edge of $G(m,2)$ lies in an induced subgraph isomorphic to $C_5$. For this, look at the cycle given by $(v_i,v_{i+1},v_{i+2},u_{i+2},u_i,v_i)$, see Figure \ref{C_5 in G(m,2)}.    
    \begin{figure}[h!]
        \begin{center}
            \begin{tikzpicture}
              \coordinate (u0) at (-1,0.5);
              \coordinate (u1) at (0,0.5);
              \coordinate (u2) at (1,0.5);
              \coordinate (u3) at (1,-0.5);
              \coordinate (u4) at (-1,-0.5);
              \draw (u0) -- (u1) -- (u2) -- (u3) -- (u4) -- cycle;
              \foreach \i in {0,...,4}
                {
                    \fill (u\i) circle (2pt);
                }
              \node[anchor=south] at (u0) {$v_i$};
              \node[anchor=south] at (u1) {$v_{i+1}$};
              \node[anchor=south] at (u2) {$v_{i+2}$};
              \node[anchor=north] at (u3) {$u_{i+2}$};
              \node[anchor=north] at (u4) {$u_i$};
            \end{tikzpicture}
            \end{center}
        \caption{$5$-cycle in $G(m,2)$}
        \label{C_5 in G(m,2)}
        \end{figure}
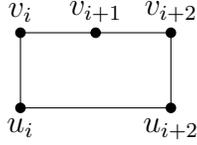
\end{proof}

The next theorem shows that metamour edges persist in $G(m,2)$ and are sorted only by parity.
\begin{theorem} \label{thm: G Metamour persistence and parity}
    Let $m,n,\ell,\ell_1,\ell_2\in\N$ with $m\geq 5$. Then 
    \[ E(\M^n(G(m,2))) \subseteq E(\M^{n+2\ell}(G(m,2)))\]
    and
    \[ E(\M^{\ell_1}(G(m,2))) \cap E(\M^{\ell_2}(G(m,2))) = \emptyset\]    
    if $\ell_1$ and $\ell_2$ have opposite parities.
\end{theorem}

\begin{proof}    
    Lemmas \ref{Expansion of 2-walks} and \ref{2-walk of length 2 in G(m,2)} show that 
    \[ E(\M^n(G(m,2))) \subseteq E(\M^{n+2\ell}(G(m,2))).\]
    From this, we see that 
    $E(\M^{\ell_1}(G(m,2))) \cap E(\M^{\ell_2}(G(m,2)))$ is empty if $\ell_1$ and $\ell_2$ have opposite parities since the two sets can be embedded into adjacent metamour powers of $G(m,2)$.
\end{proof}

In order to calculate the metamour limit set for $G(m,2)$, we continue developing our notations from Definition \ref{def:2walk}. 
\begin{definition} \label{def: Petersen fully minimal edges}
    Let $k\in\Z^+$. Define the $\emph{fully minimal $k$-set}$, 
    \[ \FM_k(G(m,2)) \subseteq E(G(m,2)) \cup E(\overline{G(m,2)}), \]
    as the set of all edges $uv$, with distinct $u,v\in V(G(m,2))$, for which there exists a fully minimal $2$-walk of length $k$ from $u$ to $v$. 
    Let
    \[ \FM_{ev}(G(m,2)) = \bigcup_{k \text{ even}} \FM_k(G(m,2)), \]
    \[ \FM_{od}(G(m,2)) = \bigcup_{k\text{ odd}} \FM_k(G(m,2)). \]
    We say that $\FM_{ev}(G(m,2))$ \emph{stabilizes by} $N\in\Z^+$ if
    \[ \FM_{ev}(G(m,2)) = \bigcup\limits_{\substack{k \text{ even} \\ k\,\le\, N}} \FM_k(G(m,2)). \]
    Similarly, $\FM_{od}(G(m,2))$ \emph{stabilizes by} $N\in\Z^+$ if
    \[ \FM_{od}(G(m,2)) = \bigcup\limits_{\substack{k \text{ odd} \\ k\,\le\, N}} \FM_k(G(m,2)). \]
    Note that $N$ is not unique.
\end{definition}
By Theorem \ref{thm: rel min walk implies edge in Mk}, we see that
\[ \FM_k(G(m,2)) \subseteq \M^k(G(m,2)). \]

Next we see that the fully minimal $k$-sets eventually capture all possible edges $uv$, $u,v\in V(G(m,2))$. 
\begin{lemma} \label{lem: Petersen all fully minimal}
    Let $m\in\Z^+$ with $m\geq 5$ and distinct $u,v\in V(G(m,2))$. Then there exists $k\in\Z^+$ such that $uv\in \FM_k(G(m,2))$.
\end{lemma}

\begin{proof} The special case $m=6$ can easily be verified by direct inspection.
    For $m\ne 6$, the proof examines, by symmetry, the three possible options for $u$ and $v$ to be either exterior or interior vertices. As the cases are similar, we give details only for the case where $u$ and $v$ are both exterior vertices. 

    First choose $\alpha\in\Z^+$, $\alpha\geq 3$, so that $2^{\alpha-1} < m \leq 2^\alpha$. After possibly relabeling, and also conflating $\Z_m$ with $\Z$ when convenient, we may assume $u=v_a$ and $v=v_b$ with $0\le a<b$ and $1\leq b-a \leq \frac{m}{2}$. For $b-a =1$, we have $uv \in \FM_2(G(m,2))$ with Lemma \ref{2-walk of length 2 in G(m,2)}, and for $b-a =2$, one easily checks $uv \in \FM_1(G(m,2))$. Thus, we can restrict ourselves to the case $3\leq b-a \leq \frac{m}{2}$. Choose $\beta\in\Z^+$, $\beta \geq 2$ so that $2^{\beta-1} < b-a \leq 2^\beta$. Note that $2^\beta < 2(b-a)\le m \le 2^\alpha$, thus $\beta\leq \alpha -1$. 

    We will look at walks $\pi$ from $v_a$ to $v_b$ of the form
    \[ (v_a, v_{a+1}, \ldots, v_{a+x},
        u_{a+x}, u_{a+x+2}, \ldots, u_{a+x+2y},
        v_{a+x+2y}, v_{a+x+2y-1}, \ldots, v_b) \]
    with $x,y\in\N$ and $a+x+2y \geq b$. Such a walk has length $x+1+y+ 1+ (a+x+2y-b) = 2x +3y +2 -(b-a)$. We will require that
    \[ 2x +3y +2 -(b-a) = 2^\beta. \]
    Note that it can be verified that $x+2y<m$ so that, in fact, $\pi$ is always a path.

    Write $c=2^\beta-2+ b-a\ge 2$. If $c\equiv 0 \mod 3$, use $x=0$ and $y=\frac{c}{3}$ and observe that this choice indeed satisfies $a+x+2y \geq b$. In this case,
    we have
    \[ \pi = (v_a,u_{a}, u_{a+2}, \ldots, u_{a+2y},
        v_{a+2y}, v_{a+2y-1}, \ldots, v_b), \]
    and we claim that $\pi$ is fully minimal. First, as $2^\beta < 2(b-a)\le m$, we have $c=2^\beta-2+ b-a < \frac 32 m$ and $2y= \frac 23 c< m$, and $\pi$ is indeed a path. Moreover, for $m\ne 6$, condition $(2)$ of Definition \ref{def:2walk} is automatically satisfied, too, and $\pi$ is a $2$-walk of length $\beta$. If we relabel $\pi$ as $(w_0,w_1,\dots,w_{2^\beta})$, it now suffices to show that there is no 
    $2$-walk from $w_{j\cdot 2^{i+1}}$ to $w_{(j+1)\cdot 2^{i+1}}$ of length $i$
    for all $1 \leq i \leq \beta-1$ and $0 \leq j \leq 2^{\beta-i-1}-1$. 

    The argument breaks into four cases depending on the location of $w_{j\cdot 2^{i+1}}$ and $w_{(j+1)\cdot 2^{i+1}}$ with respect to exterior and interior vertices. As the arguments are similar, we only give details here for two representative cases.
    
    For the first case considered here, suppose that $w_{j\cdot 2^{i+1}}$ and $w_{(j+1)\cdot 2^{i+1}}$ are both exterior vertices, neither equal to $w_0$. 
    The exterior path between these two vertices has $2^{i+1}$ edges with $2^{i+1} \leq 2y-(b-a)$. As $y=\frac{1}{3}(2^\beta-2+(b-a))$, it follows that 
    $2^{i+1} \leq \frac{2}{3}(2^\beta-2-\frac{1}{2}(b-a)) < 2^\beta$, and $i \leq \beta -2 \leq \alpha -3$.
    
    However, the shortest possible path between $w_{j\cdot 2^{i+1}}$ and $w_{(j+1)\cdot 2^{i+1}}$ would have at least either $2^i+2$ edges going along interior vertices or $\frac{m-2^{i+1}}{2}+2$ edges by going in the opposite direction. The first possibility is too large to admit a $2$-walk of length $2^i$. Turning to the second possibility, the existence of a $2$-walk of length $2^i$ would require $\frac{m-2^{i+1}}{2}+2 \leq 2^i$ so that $m \leq 2^{i+2}-4$. Thus $m < 2^{i+2}$ and $\alpha \leq i+2$, which is a contradiction. 

    For the second case considered here, suppose $w_{j\cdot 2^{i+1}}$ and $w_{(j+1)\cdot 2^{i+1}}$ are both interior vertices. 
    Then $j\geq 1$ and $2^{i+2}\leq (j+1)\cdot 2^{i+1} \leq y+1 =  \frac{1}{3}(2^\beta-2+(b-a))+1 \leq \frac{1}{3}(2^{\beta+1}+1) < 2^\beta$ so that $i\leq \beta -3 \leq \alpha -4$. However, the shortest possible path between $w_{j\cdot 2^{i+1}}$ and $w_{(j+1)\cdot 2^{i+1}}$ has either $2^{i+1}$ or $\frac{m}{2}-2^{i+1}$ edges. The first possibility is too large to allow a $2$-walk of length $2^i$. For the second possibility to work, we would need $\frac{m}{2}-2^{i+1}\leq 2^i$ so that $m\leq 2^{i+2}+2^{i+1}<2^{i+3}$. Then $\alpha \leq i+3$, which is a contradiction.
    
    Finally, the cases of $c\equiv 1\mod 3$ and $c\equiv 2\mod 3$ are done using $x=2$ and $x=1$, respectively. The details are similar and omitted.
\end{proof}

Finally, we can calculate the metamour limit period and metamour limit set of $G(m,2)$.
\begin{theorem} \label{thm: Peteresen G(m,2) limit sets}
    Let $m\in\Z^+$ with $m\geq 5$. Then $G(m,2)$ has metamour limit period $2$. 
    
    The metamour limit set consists of $(V(G(m,2)),\FM_{ev}(G(m,2)))$ and $(V(G(m,2)),\FM_{od}(G(m,2)))$, where $\FM_{ev}(G(m,2))$ and $\FM_{od}(G(m,2))$ stabilize by $2 \lfloor m/2 \rfloor +m-8$. Moreover,  
    \[\FM_{ev}(G(m,2)) = E(\M^{2\ell}(G(m,2)))\] and \[\FM_{od}(G(m,2)) = E(\M^{2\ell+1}(G(m,2)))\]
    for all sufficiently large $\ell\in\N$. Finally, 
    \[ \FM_{ev}(G(m,2)) \cup \FM_{od}(G(m,2)) = E(G(m,2)) \cup E(\overline{G(m,2)}). \]
\end{theorem}

\begin{proof}
    By Theorem \ref{thm: rel min walk implies edge in Mk}, we see that $\FM_k(G(m,2)) \subseteq E(\M^k(G(m,2)))$. By Theorem \ref{thm: G Metamour persistence and parity} and the fact that $G(m,2)$ is finite, we see that   
    \[ \FM_{ev}(G(m,2)) \subseteq E(\M^{2\ell}(G(m,2)))\] and
       \[\FM_{od}(G(m,2)) \subseteq E(\M^{2\ell+1}(G(m,2))) \]
    for all sufficiently large $\ell\in\N$.    
    From Theorem \ref{thm: G Metamour persistence and parity} and Lemma \ref{lem: Petersen all fully minimal}, we see that
    \[ E(\M^{2\ell}(G(m,2))) \cap E(\M^{2\ell+1}(G(m,2))) = \emptyset\] 
    and
    \[ \FM_{ev}(G(m,2)) \cup \FM_{od}(G(m,2)) = E(G(m,2)) \cup E(\overline{G(m,2)})\]
    so that, in fact, 
    \[ \FM_{ev}(G(m,2)) = E(\M^{2\ell}(G(m,2)))\] and
       \[\FM_{od}(G(m,2)) = E(\M^{2\ell+1}(G(m,2))) \]
    for all sufficiently large $\ell\in\N$.
    
    For the statement on stability, first observe that $G(m,2) \cup \overline{G(m,2)}$ has $\binom{2m}{2} = {m(2m-1)}$ edges, which group up into $2 \lfloor m/2 \rfloor +m$ equivalence classes with respect to index shift modulo $m$. For $m\ne 5$, $E(G(m,2))$ and $E(\M(G(m,2)))$ consist of $3$ and $6$ of these classes, respectively. Moreover, by Theorem \ref{thm: G Metamour persistence and parity}, for growing even and odd $k$, respectively, $E(\M^k(G(m,2)))$ either stays the same or grows by adding one or several of these equivalence classes. The metamour iterates will stabilize if $E(\M^k(G(m,2)))=E(\M^{k+2}(G(m,2)))$, which will happen for some $k\ge 2 \lfloor m/2 \rfloor +m-8$.
\end{proof}

We end with a characterization for the connectedness of $\M(G(m,j))$. From Definition \ref{def: M}, observe that $\M(G)$ is connected if and only if for each distinct $u,v\in V(G)$, there exists $n\in\Z^+$ and a walk in $G$, $(w_0,w_1,\ldots,w_{2n})$, with $w_{2i},w_{2i+1},w_{2(i+1)}$ distinct  and $w_{2i}w_{2(i+1)}\not\in E(V)$ for all $0\leq i \leq n-1$.

\begin{theorem} \label{thm: G(m,j) connected}
    Let $m,j \in \Z^+$ with $m \geq 5$ and $j < \frac{m}{2}$. Then $\M(G(m,j))$ is connected if and only if either
    \begin{enumerate}
        \item $m$ is odd or
        \item $m$ and $j$ are even.
    \end{enumerate}
    Otherwise, $\M(G(m,j))$ has two connected components.
\end{theorem}

\begin{proof}
    Let $u,v\in G(m,j)$ be distinct.   
    Consider first the case of odd $m$.

    If both $u,v$ are exterior vertices, then moving either in a clockwise or counterclockwise manner around the exterior vertices will furnish a path of even length showing that $u$ and $v$ are connected in $\M(G(m,j))$. It remains to show that each interior vertex is connected to an exterior vertex in $\M(G(m,j))$. For this, use the path $(u_i,u_{i+j},v_{i+j})$ for $i\in\Z_m$.

    Now consider the case of $m,j$ even. Following an argument similar to the one in the previous paragraph, it is immediate that the sets of vertices $\{v_i, u_i \,|\, i\in 2\Z_m\}$ and $\{v_i, u_i \,|\, i\in 2\Z_m+1\}$ are both connected in $\M(G(m,j))$. The path $(v_1,v_2,u_{2})$ finishes this case.

    Finally, consider the case of $m$ even and $j$ odd. Following again an argument similar to the one in the first paragraph, it is immediate that the sets of vertices $\{v_i, u_{i+1} \,|\, i\in 2\Z_m\}$ and $\{v_i, u_{i+1} \,|\, i\in 2\Z_m+1\}$ are both connected in $\M(G(m,j))$. However, as these two sets of vertices provide a $2$-coloring of $G(m,j)$, $\M(G(m,j))$ cannot be connected.
\end{proof}

\section{Metamour Graphs of Complete $m$-ary Trees}

In this section, we give a full description of the periodic behavior and limit set of the complete $m$-ary tree under the metamour operation.
From now on, for $h,m \in \Z^+$ with $m\ge 2$, we let $T \coloneqq T(h,m)$ denote the \emph{complete $m$-ary tree} with height $h$, where the \emph{height} is the number of levels below the root vertex of $T$.
Recall that the \emph{depth} of a vertex is its distance from the root.
It turns out that the central role in our analysis is played by $\M^2(T)$.
We thus begin with some auxiliary lemmas relating to this graph.

\begin{lemma}
    \label{lemma:M2 is d4}
    Let $x,y \in V(T)$.
    Then $xy \in E(\M^2(T))$ if and only if $d_T(x,y) = 4$.
\end{lemma}

\begin{proof}
    
    Suppose that $xy \in E(\M^2(T))$.
    Then by Lemma~\ref{lemma:edge in Mk implies walk}, there exists a $2$-walk $\pi = (x, u, v, w, y)$ in $T$, in the sense of Definition~\ref{def:2walk}; in particular, the vertices $x$, $v$ and $y$ are all distinct.
    Since $T$ contains no cycles, it follows that $xy \notin E(T)$ and therefore $u \neq y$ and $x \neq w$.
    Moreover, we must have $u \neq w$ because otherwise $d_T(x,y) = 2$, which means $xy \in E(\M(T))$ and thus $xy \notin E(\M^2(T))$, contradicting our supposition.
    Therefore $\pi$ is a path. Since $T$ contains no cycles, there is no shorter path from $x$ to $y$, and thus $d_T(x,y) = 4$.
    Conversely, suppose that $d_T(x,y) = 4$.
    Then there is a fully minimal $2$-walk of length $2$ from $x$ to $y$ (in the sense of Definition~\ref{def:2walk}), and so by Theorem~\ref{thm: rel min walk implies edge in Mk} we have $xy \in E(\M^2(T))$.
\end{proof}

In order to describe the relative positions of vertex pairs, we introduce the following shorthand.
Given vertices $x,y \in V(T)$,
there is a unique minimal path from $x$ to $y$, consisting of a sequence of $p$ upward steps followed by a sequence of $q$ downward steps, with $p,q \in \N$ such that $p+q = d_T(x,y)$.
We abbreviate this minimal path as
\begin{equation}
    \label{shorthand}
    \pi_T(x,y) = (x, \uparrow^p, \downarrow^q, y).
\end{equation}
For example, $x$ and $y$ are first cousins if and only if we have $\pi_T(x,y) = (x, \uparrow^2, \downarrow^2, y)$; as another example, $x$ is the great-grandchild of $y$ if and only if $\pi_T(x,y) = (x,\uparrow^3,\downarrow^0, y)$.

\begin{lemma}
    \label{lemma:4 equals 4 4s}
    Let $T = T(h,m)$ with $h \geq 5$, and let $xy \in E(\M^2(T))$.
    Then there exists $z \in V(T)$ such that both $xz$ and $yz$ are in $E(\M^3(T))$.
\end{lemma}

\begin{proof}
    We need to find $z$ such that $d_{\M^2(T)}(x,z) = d_{\M^2(T)}(y,z) = 2$.
    Equivalently, $z$ must be distinct from $x$ and $y$, and must be connected to $x$ (and separately to $y$) by concatenating two length-4 paths in $T$; moreover, we must have $d_T(x,z) \neq 4$ and $d_T(y,z) \neq 4$.
    Since by hypothesis we have $xy \in E(\M^2(T))$, Lemma~\ref{lemma:M2 is d4} implies that $d_T(x,y)=4$.
    Therefore, there are three cases (up to symmetry) which must be checked:
    \begin{enumerate}
    \setlength{\itemsep}{1ex}
        \item If $\pi_T(x,y) = (x,\uparrow^4,\downarrow^0,y)$, then we can take $z$ such that $\pi_T(x,z)$ $ = (x,\uparrow^4,\downarrow^2,z)$, as in Figure~\ref{fig:lemma 4 equals 4 4s}(A).
        \item If $\pi_T(x,y) = (x, \uparrow^3, \downarrow^1,y)$, then we can take $z$ such that $\pi_T(x,z)$ $ = (x,\uparrow^4, \downarrow^4, z)$ as depicted in Figure~\ref{fig:lemma 4 equals 4 4s}(B), as long as $x$ has depth~${\geq 4}$.
        We can also take $z$ such that $\pi_T(x,z) = (x, \uparrow^2, \downarrow^4, z)$ as depicted in Figure~\ref{fig:lemma 4 equals 4 4s}(C), as long as $x$ has depth $\leq h-2$. 
        The tree $T(h,m)$ admits at least one of these two possibilities if and only if $h \geq 5$.

        \item If $\pi_T(x,y) = (x, \uparrow^2, \downarrow^2, y)$, then we can take $z$ such that $\pi_T(x,z) = (x, \uparrow^4, \downarrow^2, z)$ as in Figure~\ref{fig:lemma 4 equals 4 4s}(D), as long as $x$ and $y$ have depth~${\geq 4}$.
        We can also take $z$ such that $\pi_T(x,z) = (x, \uparrow^2, \downarrow^0, z)$ as in Figure~\ref{fig:lemma 4 equals 4 4s}(E), as long as $x$ and $y$ have depth $\leq h-2$.
        The tree $T(h,m)$ admits at least one of these possibilities if and only if $h \geq 5$. \qedhere
        \end{enumerate}
\end{proof}

\begin{figure}[h!]
    \centering
    \begin{subfigure}[b]{0.3\textwidth}
        \centering

\begin{center}
\resizebox{1\textwidth}{!}{%
\trimbox{0cm 0cm 0cm 0cm}{ 
\begin{tikzpicture}[level distance=10mm, state/.style={circle, draw, minimum size=1mm}]]
  \tikzstyle{every node}=[font=\Large, draw=black,circle,inner sep=3pt, minimum size=1mm]
  \tikzstyle{level 1}=[sibling distance=25mm, set style={{every node}+=[draw=black]}]
  \tikzstyle{level 2}=[sibling distance=15mm, set style={{every node}+=[draw=black]}]
  \tikzstyle{level 3}=[sibling distance=15mm,  set style={{every node}+=[draw=black]}]
  \tikzstyle{level 4}=[sibling distance=15mm,  set style={{every node}+=[draw=black]}]
  \tikzstyle{level 5}=[sibling distance=15mm,  set style={{every node}+=[draw=black]}]

    \node[label={$y$}] (y) {}
     child {node (1) {}
      child {node (2) {}
       child {node (3) {}
        child {node[label=below:{$x$}] (x) {}}
        child[fill=none] {edge from parent[draw=none]}
       }
       child[fill=none] {edge from parent[draw=none]
        child[fill=none] {edge from parent[draw=none]}
        child[fill=none] {edge from parent[draw=none]}
       }
      }
      child {node[label=below:{$a$}] (a) {}
       child[fill=none] {edge from parent[draw=none]
        child[fill=none] {edge from parent[draw=none]}
        child[fill=none] {edge from parent[draw=none]}
        }
       child[fill=none] {edge from parent[draw=none]
        child[fill=none] {edge from parent[draw=none]}
        child[fill=none] {edge from parent[draw=none]}
       }
      }
      }
      child {node (4) {}
       child {node[label=below:{$z$}] (z) {}
        child[fill=none] {edge from parent[draw=none]
         child[fill=none] {edge from parent[draw=none]}
         child[fill=none] {edge from parent[draw=none]}
         }
       child[fill=none] {edge from parent[draw=none]
        child[fill=none] {edge from parent[draw=none]}
        child[fill=none] {edge from parent[draw=none]}
       }
      }
       child {node (5) {}
        child[fill=none] {edge from parent[draw=none]
         child[fill=none] {edge from parent[draw=none]}
         child[fill=none] {edge from parent[draw=none]}
        }
       child {node (6) {}
        child[fill=none] {edge from parent[draw=none]}
        child {node[label=below:{$b$}] (b) {}}
        }
       }
       };

    \draw[-, dotted,ultra thick,relative,red] (a) to[out=-60,in=-200] (x);

    \draw[-, dotted,ultra thick,relative,red] (a) to[out=120,in=60, looseness=3.5] (z);

    \draw[-, solid,ultra thick,relative,blue] (z) to[out=60,in=200] (b);
  
    \draw[-, solid,ultra thick,relative,blue] (y) to[out=40,in=-220] (b);

\end{tikzpicture}
}
}
\end{center}
        \caption{}
        \label{subfig:lemma73 1}
    \end{subfigure}
        \begin{subfigure}[b]{0.3\textwidth}
        \centering

\begin{center}
\resizebox{1\textwidth}{!}{%
\trimbox{0cm 0cm 0cm 0cm}{ 
\begin{tikzpicture}[level distance=10mm, state/.style={circle, draw, minimum size=1mm}]]
  \tikzstyle{every node}=[font=\Large, draw=black,circle,inner sep=3pt, minimum size=1mm]
  \tikzstyle{level 1}=[sibling distance=25mm, set style={{every node}+=[draw=black]}]
  \tikzstyle{level 2}=[sibling distance=15mm, set style={{every node}+=[draw=black]}]
  \tikzstyle{level 3}=[sibling distance=15mm,  set style={{every node}+=[draw=black]}]
  \tikzstyle{level 4}=[sibling distance=15mm,  set style={{every node}+=[draw=black]}]
  \tikzstyle{level 5}=[sibling distance=15mm,  set style={{every node}+=[draw=black]}]

    \node[label={$a$}] (a) {}
     child {node (1) {}
      child {node (2) {}
       child {node (3) {}
        child {node[label=below:{$x$}] (x) {}}
        child[fill=none] {edge from parent[draw=none]}
       }
       child[fill=none] {edge from parent[draw=none]
        child[fill=none] {edge from parent[draw=none]}
        child[fill=none] {edge from parent[draw=none]}
       }
      }
      child {node[label=below:{$y$}] (y) {}
       child[fill=none] {edge from parent[draw=none]
        child[fill=none] {edge from parent[draw=none]}
        child[fill=none] {edge from parent[draw=none]}
        }
       child[fill=none] {edge from parent[draw=none]
        child[fill=none] {edge from parent[draw=none]}
        child[fill=none] {edge from parent[draw=none]}
       }
      }
      }
      child {node (4) {}
       child {node[label=below:{$b$}] (b) {}
        child[fill=none] {edge from parent[draw=none]
         child[fill=none] {edge from parent[draw=none]}
         child[fill=none] {edge from parent[draw=none]}
         }
       child[fill=none] {edge from parent[draw=none]
        child[fill=none] {edge from parent[draw=none]}
        child[fill=none] {edge from parent[draw=none]}
       }
      }
       child {node (5) {}
        child[fill=none] {edge from parent[draw=none]
         child[fill=none] {edge from parent[draw=none]}
         child[fill=none] {edge from parent[draw=none]}
        }
       child {node (6) {}
        child[fill=none] {edge from parent[draw=none]}
        child {node[label=below:{$z$}] (z) {}}
        }
       }
       };

    \draw[-, dotted,ultra thick,relative,red] (x) to[out=40,in=-220] (a);

    \draw[-, dotted,ultra thick,relative,red] (a) to[out=40,in=-220] (z);

    \draw[-, solid,ultra thick,relative,blue] (y) to[out=120,in=60, looseness=3.5] (b);

    \draw[-, solid,ultra thick,relative,blue] (b) to[out=60,in=200] (z);

\end{tikzpicture}
}
}
\end{center}
        \caption{}
        \label{subfig:lemma73 2a}
    \end{subfigure}
        \begin{subfigure}[b]{0.3\textwidth}
        \centering
        \begin{center}
\resizebox{1\textwidth}{!}{%
\trimbox{0cm 0cm 2cm 0cm}{ 
\begin{tikzpicture}[level distance=10mm, state/.style={circle, draw, minimum size=1mm}]]
  \tikzstyle{every node}=[font=\large, draw=black,circle,inner sep=3pt]
  \tikzstyle{level 1}=[sibling distance=20mm, set style={{every node}+=[draw=black]}]
  \tikzstyle{level 2}=[sibling distance=20mm, set style={{every node}+=[draw=black]}]
  \tikzstyle{level 3}=[sibling distance=20mm,  set style={{every node}+=[draw=black]}]
  \tikzstyle{level 4}=[sibling distance=20mm,  set style={{every node}+=[draw=black]}]
  \tikzstyle{level 5}=[sibling distance=20mm,  set style={{every node}+=[draw=black]}]
  \tikzstyle{level 6}=[sibling distance=20mm,  set style={{every node}+=[draw=black]}]

  \node (1) {}
     child {node (2) {}
      child {node (3) {}
       child {node[label=below:{$x$}] (x) {}
        child[fill=none] {edge from parent[draw=none]}
        child[fill=none] {edge from parent[draw=none]}
       }
       child[fill=none] {edge from parent[draw=none]
        child[fill=none] {edge from parent[draw=none]}
        child[fill=none] {edge from parent[draw=none]}
       }
      }
      child {node (4) {}
       child {node (5) {}
        child {node (6) {}
         child {node[label=below:{$z$}] (z) {}}
         child[fill=none] {edge from parent[draw=none]}
        }
        child[fill=none] {edge from parent[draw=none]}
        }
       child {node[label=right:{$a=b$}] (a) {}
        child[fill=none] {edge from parent[draw=none]}
        child[fill=none] {edge from parent[draw=none]}
       }
      }
      }
      child {node[label=below:{$y$}] (y) {}
       child[fill=none] {edge from parent[draw=none]
        child[fill=none] {edge from parent[draw=none]
         child[fill=none] {edge from parent[draw=none]}
         child[fill=none] {edge from parent[draw=none]}
         }
       child[fill=none] {edge from parent[draw=none]
        child[fill=none] {edge from parent[draw=none]}
        child[fill=none] {edge from parent[draw=none]}
       }
      }
       child[fill=none] {edge from parent[draw=none]
        child[fill=none] {edge from parent[draw=none]
         child[fill=none] {edge from parent[draw=none]}
         child[fill=none] {edge from parent[draw=none]}
        }
       child[fill=none] {edge from parent[draw=none]
        child[fill=none] {edge from parent[draw=none]}
        child[fill=none] {edge from parent[draw=none]}
        }
       }
       };

       \draw[-, dotted,ultra thick,relative,red] (x) to[out=40,in=140, looseness=1.4] (a);

       \draw[-, dotted,ultra thick,relative,red] (z) to[out=-5,in=120, looseness=1] (a);
 
       \draw[-, solid,ultra thick,relative,blue] (z) to[out=-15,in=140, looseness=1] (a);

       \draw[-, solid,ultra thick,relative,blue] (y) to[out=235,in=215, looseness=2] (a);

\end{tikzpicture}
}
}
\end{center}
        \caption{}
        \label{subfig:lemma73 2b}
    \end{subfigure}
        \begin{subfigure}[b]{0.3\textwidth}
        \centering

\begin{center}
\resizebox{1\textwidth}{!}{%
\trimbox{0cm 0cm 1.5cm 0cm}{ 
\begin{tikzpicture}[level distance=10mm, state/.style={circle, draw, minimum size=1mm}]]
  \tikzstyle{every node}=[font=\Large, draw=black, circle,inner sep=3pt]
  \tikzstyle{level 1}=[sibling distance=25mm, set style={{every node}+=[draw=black]}]
  \tikzstyle{level 2}=[sibling distance=25mm, set style={{every node}+=[draw=black]}]
  \tikzstyle{level 3}=[sibling distance=25mm,  set style={{every node}+=[draw=black]}]
  \tikzstyle{level 4}=[sibling distance=12mm,  set style={{every node}+=[draw=black]}]
  \tikzstyle{level 5}=[sibling distance=5mm,  set style={{every node}+=[draw=black]}]

  \node (1) {}
     child {node (2) {}
      child {node (3) {}
       child {node (4) {}
        child {node[label=below:{$x$}] (x) {}}
        child {node (5) {}}
       }
       child {node (6) {}
        child {node (7) {}}
        child {node[label=below:{$y$}] (y) {}}
       }
      }
      child {node[label=right:{$a=b$}] (a) {}
       child[fill=none] {edge from parent[draw=none]
        child[fill=none] {edge from parent[draw=none]}
        child[fill=none] {edge from parent[draw=none]}
        }
       child[fill=none] {edge from parent[draw=none]
        child[fill=none] {edge from parent[draw=none]}
        child[fill=none] {edge from parent[draw=none]}
       }
      }
      }
      child {node (9) {}
       child[fill=none] {edge from parent[draw=none]
        child[fill=none] {edge from parent[draw=none]
         child[fill=none] {edge from parent[draw=none]}
         child[fill=none] {edge from parent[draw=none]}
         }
       child[fill=none] {edge from parent[draw=none]
        child[fill=none] {edge from parent[draw=none]}
        child[fill=none] {edge from parent[draw=none]}
       }
      }
       child {node[label=below:{$z$}] (z) {}
        child[fill=none] {edge from parent[draw=none]
         child[fill=none] {edge from parent[draw=none]}
         child[fill=none] {edge from parent[draw=none]}
        }
       child[fill=none] {edge from parent[draw=none]
        child[fill=none] {edge from parent[draw=none]}
        child[fill=none] {edge from parent[draw=none]}
        }
       }
       };

       \draw[-, solid,ultra thick,relative,blue] (x) to[out=15,in=120, looseness=1] (a);

       \draw[-, solid,ultra thick,relative,blue] (a) to[out=120,in=150, looseness=2] (z);


       \draw[-, dotted,ultra thick,relative,red] (y) to[out=35,in=90, looseness=1.4] (a);

       \draw[-, dotted,ultra thick,relative,red] (a) to[out=113,in=158, looseness=1.8] (z);

\end{tikzpicture}
}
}
\end{center}
        \caption{}
        \label{subfig:lemma73 3a}
    \end{subfigure}
        \begin{subfigure}[b]{0.3\textwidth}
        \centering

\begin{center}
\resizebox{1\textwidth}{!}{%
\trimbox{0cm 0cm 0cm 0cm}{ 
\begin{tikzpicture}[level distance=10mm, state/.style={circle, draw, minimum size=1mm}]
  \tikzstyle{every node}=[font=\Large, draw=black,circle,inner sep=3pt, minimum size=1mm]
  \tikzstyle{level 1}=[sibling distance=25mm, set style={{every node}+=[draw=black]}]
  \tikzstyle{level 2}=[sibling distance=15mm, set style={{every node}+=[draw=black]}]
  \tikzstyle{level 3}=[sibling distance=15mm,  set style={{every node}+=[draw=black]}]
  \tikzstyle{level 4}=[sibling distance=15mm,  set style={{every node}+=[draw=black]}]
  \tikzstyle{level 5}=[sibling distance=15mm,  set style={{every node}+=[draw=black]}]

    \node[label={$z$}] (z) {}
     child {node (1) {}
      child {node (2) {}
       child {node (3) {}
        child {node[label=below:{$a$}] (a) {}}
        child[fill=none] {edge from parent[draw=none]}
       }
       child[fill=none] {edge from parent[draw=none]
        child[fill=none] {edge from parent[draw=none]}
        child[fill=none] {edge from parent[draw=none]}
       }
      }
      child {node[label=below:{$x$}] (x) {}
       child[fill=none] {edge from parent[draw=none]
        child[fill=none] {edge from parent[draw=none]}
        child[fill=none] {edge from parent[draw=none]}
        }
       child[fill=none] {edge from parent[draw=none]
        child[fill=none] {edge from parent[draw=none]}
        child[fill=none] {edge from parent[draw=none]}
       }
      }
      }
      child {node (4) {}
       child {node[label=below:{$y$}] (y) {}
        child[fill=none] {edge from parent[draw=none]
         child[fill=none] {edge from parent[draw=none]}
         child[fill=none] {edge from parent[draw=none]}
         }
       child[fill=none] {edge from parent[draw=none]
        child[fill=none] {edge from parent[draw=none]}
        child[fill=none] {edge from parent[draw=none]}
       }
      }
       child {node (5) {}
        child[fill=none] {edge from parent[draw=none]
         child[fill=none] {edge from parent[draw=none]}
         child[fill=none] {edge from parent[draw=none]}
        }
       child {node (6) {}
        child[fill=none] {edge from parent[draw=none]}
        child {node[label=below:{$b$}] (b) {}}
        }
       }
       };

    \draw[-, solid,ultra thick,relative,blue] (x) to[out=-60,in=-200] (a);

    \draw[-, solid,ultra thick,relative,blue] (a) to[out=40,in=-220] (z);

    \draw[-, dotted,ultra thick,relative,red] (z) to[out=40,in=-220] (b);

    \draw[-, dotted,ultra thick,relative,red] (y) to[out=60,in=200] (b);

\end{tikzpicture}
}
}
\end{center}
        \caption{}
        \label{subfig:lemma73 3b}
    \end{subfigure}

    \caption{Illustrations of the proof of Lemma~\ref{lemma:4 equals 4 4s}.
    In each case, $xz$ and $yz$ are edges in $\M^3(T)$.
    The diagrams above show this by exhibiting a path $(x,a,z)$ in $\M^2(T)$, depicted by solid arcs, and a path $(y,b,z)$ in $\M^2(T)$, depicted by dotted arcs.
    Note that we depict the case $m=2$ (i.e., where $T$ is a binary tree), but the situation is the same for any $m \geq 2$.}
    \label{fig:lemma 4 equals 4 4s}
\end{figure}
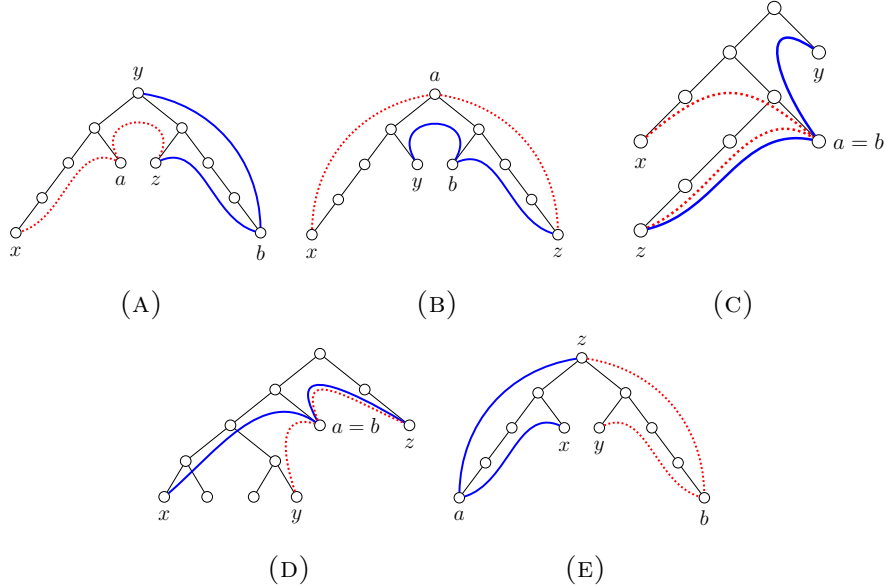

\begin{lemma}
    \label{lemma:M3T}
    Let $T = T(h,m)$, where $h \geq 5$.
    Suppose ${d_{\M^2(T)}(x,y) = 3}$. Then there exists $z \in V(T)$ such that both $xz$ and $yz$ are in $E(\M^3(T))$ unless

\begin{center}
$x$ and $y$ both have depth $1$ in $T$ with $m = 2$ and $h \in \{5,6\}$.
\end{center}

        

        
    
\end{lemma}

\begin{proof}
We need to find a vertex $z$ with the same properties described in the proof of Lemma~\ref{lemma:4 equals 4 4s}.
    Up to symmetry, there are 15 relative positions for $x$ and $y$ such that $d_{\M^2(T)}(x,y) = 3$, and such that $x$ and $y$ do not both have depth 1 in $T$.
    (To determine these 15 positions, one need only check vertex pairs whose distance is a positive even integer which is at most $3 \times 4 = 12$.)
    In Table~\ref{table:proofM3T}, we exhibit a choice of the desired vertex $z$ for each of these 15 positions, using the shorthand in~\eqref{shorthand}.
    (Note that the vertex $z$ described in the table may not be unique; rather, any $z$ that satisfies the location given in the table has the desired property.)
    In each case, it is straightforward to verify (via diagrams like those in Figure~\ref{fig:lemma 4 equals 4 4s}) that both $xz$ and $yz$ belong to $E(\M^3(T))$.

    Suppose now that $x$ and $y$ both have depth $1$ in $T$.
    If $m > 2$, then we observe that $d_{\M^2(T)}(x,y) = 2$. If $m = 2$ and $h>6$, then we can take $z$ such that $\pi_T(x,z)$ $ = (x,\uparrow^1, \downarrow^7, z)$.
    If, however, $m=2$ and $h \in \{5,6\}$, then $d_{\M^2(T)}(x,y) = 3$, and it is straightforward to verify that there is no vertex $z$ with the desired property.
\end{proof}

\begin{table}[h!]
    \centering
    \begin{tabular}{|c|c|}
    \hline
        $\pi_T(x,y) = (x, \uparrow^p, \downarrow^q, y)$ & $\pi_T(x,z) = (x, \uparrow^r, \downarrow^s, z)$ \\
    \hline
        $p=12, \; q=0$ & \multirow{6}{*}{$r=6, \; s=2$} \\
        \cline{1-1}
    $p=11, \; q=1$ &  \\
    \cline{1-1}
    $p=10, \; q=2$ &  \\
    \cline{1-1}
    $p=9, \; q=3$ &  \\
    \cline{1-1}
    $p=8, \; q=4$ &  \\
    \cline{1-1}
    $p=7, \; q=5$ &  \\
    \hline
    $p=6, \; q=6$ & $r = 6, \: s = 0$ \\
    \hline
    $p=10, \; q=0$ & \multirow{6}{*}{$r=4, \; s=2$} \\
    \cline{1-1}
    $p=9, \; q=1$ &  \\
    \cline{1-1}
    $p=8, \; q=2$ &  \\
    \cline{1-1}
    $p=7, \; q=3$ &  \\
    \cline{1-1}
    $p=6, \; q=4$ &  \\
    \cline{1-1}
    $p=5, \; q=5$ &  \\
    \hline
    $p=3, \; q=3$ (if $\text{depth $x =3$}$) & $r = 2, \: s = 4$ \\
    \hline
    $p=2, \; q=0$ (if $\text{depth $x \geq h-1$}$) & $r = 4, \: s = 4$ \\
    \hline
    \end{tabular}
    \caption{Case-by-case proof of 
    Lemma~\ref{lemma:M3T}, using the shorthand in~\eqref{shorthand}.
    The first column gives the relative positions of $x$ and $y$, and the second column exhibits the vertex~$z$ (described relative to $x$) referred to in the lemma.}
    \label{table:proofM3T}
\end{table}

\begin{lemma}
    \label{lemma:diam M2T}
     Let $T = T(h,m)$ with $h \geq 5$.
    Then $\M^2(T)$ has two connected components, namely the vertices with even depth and the vertices with odd depth.
    The maximum of the diameters of these two connected components is $\lceil h/2 \rceil$.
\end{lemma}

\begin{proof}
    It is clear that if $d_T(x,y) = 4$, then the depths of $x$ and $y$ have the same parity.
    Thus by Lemma~\ref{lemma:M2 is d4}, if $xy \in E(\M^2(T))$ then the depth of $x$ and the depth of $y$ have the same parity.
    
    Conversely, we claim that any two vertices $x$ and $y$, whose depths have the same parity, are connected in $\M^2(T)$.
    To see this, observe that because $T$ is a tree, there is a unique path in $T$ between $x$ and $y$, which necessarily has even length.
    If this length is divisible by 4, then we are done by Lemma~\ref{lemma:M2 is d4}.
    Thus, let us assume that the length is $4\ell+2$ for some $\ell \in \N$, and that the path is given by $(x,w_1,w_2,\ldots,w_{4\ell+1},y)$. For $\ell>0$, if $w_{4\ell-1}$ is distinct from the root of the tree $T$, let $z$ denote any neighbor of $w_{4\ell-1}$ distinct from both $w_{4\ell-2}$ and $w_{4\ell}$. Then,
    $$(x,w_1,w_2,\ldots,w_{4\ell-2},w_{4\ell-1},z,w_{4\ell-1},w_{4\ell},w_{4\ell+1},y)$$
    is a walk in $T$ of length $4\ell$ for which, starting with the vertex $x$, every fourth vertex is distance $4$ from the previous vertex. Thus, the path $(x,w_4,w_8,\ldots,w_{4\ell-4},z,y)$ connects $x$ and $y$ in $\M^2(T)$. If $w_{3}$ is distinct from the root of the tree $T$, we can make a similar argument to show that $x$ and $y$ are connected in $\M^2(T)$. This leaves us to consider the following special cases:

        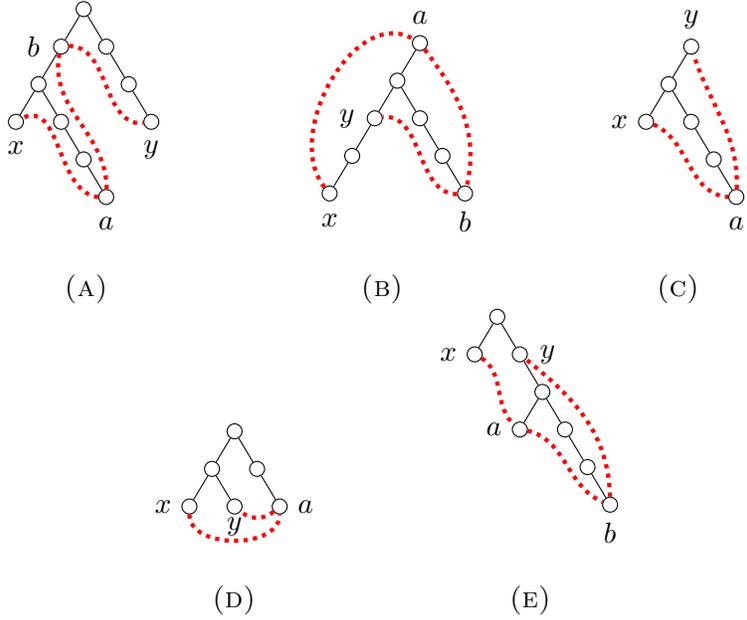
\begin{figure}[h!]
    \centering
    \begin{subfigure}[b]{0.3\textwidth}
        \centering
        \begin{center}
\begin{tikzpicture}[scale=.5,level distance=10mm, state/.style={circle, draw, minimum size=1mm}]]
  \tikzstyle{every node}=[font=\small, draw=black,circle,inner sep=2pt, minimum size=1mm]
  \tikzstyle{level 1}=[sibling distance=12mm, set style={{every node}+=[draw=black]}]
  \tikzstyle{level 2}=[sibling distance=12mm, set style={{every node}+=[draw=black]}]
  \tikzstyle{level 3}=[sibling distance=12mm,  set style={{every node}+=[draw=black]}]
  \tikzstyle{level 4}=[sibling distance=12mm,  set style={{every node}+=[draw=black]}]
  \tikzstyle{level 5}=[sibling distance=12mm,  set style={{every node}+=[draw=black]}]
  \tikzstyle{level 6}=[sibling distance=12mm,  set style={{every node}+=[draw=black]}]

    \node (u) {}
     child {
        node[label=left:{$b$}] (b) {}
              child {node (B) {}
                child {node (x)[label=below:{$x$}] {}
       }
       child {node (F) {}
        child[fill=none] {edge from parent[draw=none]}
        child {node (G) {}
            child[fill=none] {edge from parent[draw=none]}
            child {node[label=below:{$a$}] (a) {}}
            }
      }}
      child[fill=none] {edge from parent[draw=none]
       child[fill=none] {edge from parent[draw=none]
        child[fill=none] {edge from parent[draw=none]}
        child[fill=none] {edge from parent[draw=none]}
        }
       child[fill=none] {edge from parent[draw=none]
        child[fill=none] {edge from parent[draw=none]}
        child[fill=none] {edge from parent[draw=none]}
       }
      }
      }
      child {node{}
       child[fill=none] {edge from parent[draw=none]
        child[fill=none] {edge from parent[draw=none]
         child[fill=none] {edge from parent[draw=none]}
         child[fill=none] {edge from parent[draw=none]}
         }
       child[fill=none] {edge from parent[draw=none]
        child[fill=none] {edge from parent[draw=none]}
        child[fill=none] {edge from parent[draw=none]}
       }
      }
       child[fill=none] {node {}
        child[fill=none] {edge from parent[draw=none]
         child[fill=none] {edge from parent[draw=none]}
         child[fill=none] {edge from parent[draw=none]}
        }
       child[fill=none] {node[label=below:{$y$}] (y) {}
        child[fill=none] {edge from parent[draw=none]}
        child[fill=none] {edge from parent[draw=none]}
        }
       }
       };

    \draw[-,dotted,ultra thick,relative,red] (a) to[out=40,in=250] (x);
    
    \draw[-,dotted,ultra thick,relative,red] (a) to[out=-20,in=150] (b);

    \draw[-,dotted,ultra thick,relative,red] (b) to[out=45,in=230] (y);
       
\end{tikzpicture}
\end{center}
        \caption{}
        \label{subfig:lemma94 1}
    \end{subfigure}
        \begin{subfigure}[b]{0.3\textwidth}
        \centering
        \begin{center}
\begin{tikzpicture}[scale=.5,level distance=10mm, state/.style={circle, draw, minimum size=1mm}]]
  \tikzstyle{every node}=[font=\small, draw=black,circle,inner sep=2pt, minimum size=1mm]
  \tikzstyle{level 1}=[sibling distance=12mm, set style={{every node}+=[draw=black]}]
  \tikzstyle{level 2}=[sibling distance=12mm, set style={{every node}+=[draw=black]}]
  \tikzstyle{level 3}=[sibling distance=12mm,  set style={{every node}+=[draw=black]}]
  \tikzstyle{level 4}=[sibling distance=12mm,  set style={{every node}+=[draw=black]}]
  \tikzstyle{level 5}=[sibling distance=12mm,  set style={{every node}+=[draw=black]}]

    \node [label=above:{$a$}] (a) {}
     child {
        node {}
              child {node (y)[label=left:{$y$}] {}
                child {node {} 
                child{node[label=below:{$x$}] (x) {}}
                child[fill=none] {edge from parent[draw=none]}
                }
                child[fill=none] {edge from parent[draw=none]}
                }
                child {node {}
                child[fill=none] {edge from parent[draw=none]}
                child {node {}
                child[fill=none] {edge from parent[draw=none]}
                child {node[label=below:{$b$}] (b) {}}}
                }
            }
    child[fill=none] {edge from parent[draw=none]};
    
    \draw[-,dotted,ultra thick,relative,red] (x) to[out=70,in=90] (a);
    
    \draw[-,dotted,ultra thick,relative,red] (a) to[out=20,in=150] (b);

    \draw[-,dotted,ultra thick,relative,red] (b) to[out=45,in=230] (y);
       
\end{tikzpicture}
\end{center}
        \caption{}
        \label{subfig:lemma94 2a}
    \end{subfigure}
        \begin{subfigure}[b]{0.3\textwidth}
        \centering
        \begin{center}
\begin{tikzpicture}[scale=.5,level distance=10mm, state/.style={circle, draw, minimum size=1mm}]]
  \tikzstyle{every node}=[font=\small, draw=black,circle,inner sep=2pt, minimum size=1mm]
  \tikzstyle{level 1}=[sibling distance=12mm, set style={{every node}+=[draw=black]}]
  \tikzstyle{level 2}=[sibling distance=12mm, set style={{every node}+=[draw=black]}]
  \tikzstyle{level 3}=[sibling distance=12mm,  set style={{every node}+=[draw=black]}]
  \tikzstyle{level 4}=[sibling distance=12mm,  set style={{every node}+=[draw=black]}]
  \tikzstyle{level 5}=[sibling distance=12mm,  set style={{every node}+=[draw=black]}]

    \node [label=above:{$y$}] (y) {}
     child {
        node {}
              child {node (x)[label=left:{$x$}] {}}
                child {node {}
                child[fill=none] {edge from parent[draw=none]}
                child {node {}
                child[fill=none] {edge from parent[draw=none]}
                child {node[label=below:{$a$}] (a) {}}}
                }
            }
    child[fill=none] {edge from parent[draw=none]};
    
    \draw[-,dotted,ultra thick,relative,red] (x) to[out=20,in=220] (a);
    
    \draw[-,dotted,ultra thick,relative,red] (a) to[out=-20,in=180] (y);
       
\end{tikzpicture}
\end{center}
        \caption{}
        \label{subfig:lemma94 2b}
    \end{subfigure}
        \begin{subfigure}[b]{0.3\textwidth}
        \centering
        \begin{center}
\begin{tikzpicture}[scale=.5,level distance=10mm, state/.style={circle, draw, minimum size=1mm}]]
  \tikzstyle{every node}=[font=\small, draw=black,circle,inner sep=2pt, minimum size=1mm]
  \tikzstyle{level 1}=[sibling distance=12mm, set style={{every node}+=[draw=black]}]
  \tikzstyle{level 2}=[sibling distance=12mm, set style={{every node}+=[draw=black]}]
  \tikzstyle{level 3}=[sibling distance=12mm,  set style={{every node}+=[draw=black]}]

    \node {}
     child {
        node {}
              child {node [label=left:{$x$}] (x) {}}
                child {node [label={[label distance=-1mm]below:{$y$}}] (y) {}}}
            child{ node {}
            child[fill=none] {edge from parent[draw=none]}
            child{node[label=right:{$a$}] (a) {}}};
    
    \draw[-,dotted,ultra thick,relative,red] (x) to[out=270,in=270] (a);
    
    \draw[-,dotted,ultra thick,relative,red] (a) to[out=40,in=-220] (y);
       
\end{tikzpicture}
\end{center}
        \caption{}
        \label{subfig:lemma94 3a}
    \end{subfigure}
        \begin{subfigure}[b]{0.3\textwidth}
        \centering
        \begin{center}
\begin{tikzpicture}[scale=.5,level distance=10mm, state/.style={circle, draw, minimum size=1mm}]]
  \tikzstyle{every node}=[font=\small, draw=black,circle,inner sep=2pt, minimum size=1mm]
  \tikzstyle{level 1}=[sibling distance=12mm, set style={{every node}+=[draw=black]}]
  \tikzstyle{level 2}=[sibling distance=12mm, set style={{every node}+=[draw=black]}]
  \tikzstyle{level 3}=[sibling distance=12mm,  set style={{every node}+=[draw=black]}]

    \node {}
     child {
        node[label=left:{$x$}] (x) {}}
            child{ node[label=right:{$y$}] (y) {}
            child[fill=none] {edge from parent[draw=none]}
            child{node {}
            child{node[label=left:{$a$}] (a) {}}
            child{node {}
            child[fill=none] {edge from parent[draw=none]}
            child{node {}
            child[fill=none] {edge from parent[draw=none]}
            child{node[label=below:{$b$}] (b) {}
            }
            }}
            }
            };
    
    \draw[-,dotted,ultra thick,relative,red] (x) to[out=30,in=200] (a);

    \draw[-,dotted,ultra thick,relative,red] (a) to[out=30,in=200] (b);
    
    \draw[-,dotted,ultra thick,relative,red] (b) to[out=-30,in=200] (y);
       
\end{tikzpicture}
\end{center}
        \caption{}
        \label{subfig:lemma94 3b}
    \end{subfigure}

    \caption{Illustrations of the proof of Lemma~\ref{lemma:diam M2T}.
    In each case, the dotted arcs represent edges in $\M^2(T)$ that connect $x$ with $y$ via a path $(x,a,y)$ or $(x,a,b,y)$, respectively.}
    \label{fig:lemma 9.4}
\end{figure}

    \begin{enumerate}
        \item Suppose that $\ell=1$, and that $w_3$ is the root of the tree $T$. In this situation, $x$ and $y$ have both depth $3$ and are connected in $\M^2(T)$ as depicted in Figure~\ref{fig:lemma 9.4}(A).
        \item Suppose that $\ell=0$, and that $x$ and $y$ have distinct depths. Without loss of generality, let us assume that $x$ has the larger depth. Then, as long as $x$ has depth $\geq 4$, $x$ and $y$ are connected in $\M^2(T)$ as depicted in Figure~\ref{fig:lemma 9.4}(B). If $x$ has depth $\le h-2$, an alternative connecting path is shown in Figure~\ref{fig:lemma 9.4}(C).
        \item Suppose that $\ell=0$, and that $x$ and $y$ have the same depth. Then, as long as $x$ has depth $\geq 2$, $x$ and $y$ are connected in $\M^2(T)$ as depicted in Figure~\ref{fig:lemma 9.4}(D). If $x$ has depth $\le h-4$, an alternative connecting path is shown in Figure~\ref{fig:lemma 9.4}(E).
    \end{enumerate}
    Hence $\M^2(T)$ has exactly two connected components.
    
    The argument above also shows that if $x$ and $y$ are in the same connected component with $ d_T(x,y)\ge 8$, then $d_{\M^2(T)}(x,y) = \lceil d_T(x,y)/4 \rceil$.
    Therefore, since ${\rm diam}(T) = 2h$, the maximum distance between connected vertices in $\M^2(T)$ is $\lceil h/2 \rceil$.
\end{proof}


To streamline the statement of our results, we partition the positive integers into segments $S_i$ whose endpoints are consecutive powers of $2$:
\begin{equation}
    \label{Si}
    S_i \coloneqq (2^{i-1}, 2^i], \qquad i = 0, 1, 2, \ldots,
\end{equation}
where we use the standard interval notation restricted to the integers.
In other words, we have
\begin{equation}
\label{Si log def}
    S_i = \{ s \in \Z^+ \mid \lceil \log_2 s \rceil = i\}.
\end{equation}
Note that $S_0 = \{1\}$ and $S_1 = \{2\}$.

\begin{lemma}
\label{lemma:combinations}
    Let $i\in \Z^+$ with $i \neq 2$.
    If $i$ is even $($resp., odd$)$, then every element of $S_i$ can be written as the sum of two $($not necessarily distinct$)$ elements in the union of $S_1$ $($resp., $S_0)$ and $S_{i-1}$.
\end{lemma}

\begin{proof}
    The statement is obviously true for $i=1$. Thus let $i \in \Z^+$ with $i\ge 3$.
    We have $S_i = [2^{i-1} + 1, \: 2^i]$ and $S_{i-1} = [2^{i-2}+1, \: 2^{i-1}]$.
    Since the sumset of $S_{i-1}$ is 
    \[
    \{s+t \mid s,t \in S_{i-1}\} = [2\min S_{i-1}, 2\max S_{i-1}] = [2^{i-1} + 2, 2^i],
    \]
    we have proved the lemma for all elements of $S_i$ except for $2^{i-1}+1$.
    Now, if $i$ is even, we complete the proof by writing $2^{i-1} + 1$ as the sum of $2 \in S_1$ and $2^{i-1}-1 \in S_{i-1}$. 
    Likewise, if $i$ is odd, we complete the proof by writing $2^{i-1} + 1$ as the sum of $1 \in S_0$ and $2^{i-1} \in S_{i-1}$.
\end{proof}

\begin{corollary}
    \label{cor:combinations}
    Let $\ell\in \Z^+$.
    Then every element of
    \[
    \bigcup_{\substack{0\,\le\, i\, \leq\, \ell \\ i\, \equiv\, \ell \: ({\rm mod} \: 2)}} \hspace{-2.5ex} S_i,
    \]
    except for $1$ and $3$ $($which appear only when $\ell$ is even$)$, can be written as the sum of two $($not necessarily distinct$)$ elements of 
    \[
    \bigcup_{\substack{0\,\le\, i\, \leq\, \ell-1 \\ i\, \equiv\, \ell-1 \: ({\rm mod} \: 2)}} \hspace{-2.5ex} S_i.
    \]
\end{corollary}

\begin{proof}
    This follows immediately from Lemma~\ref{lemma:combinations}. 
    Note that the sets $S_0=\{1\}$ and $S_2=\{3,4\}$ are excluded from Lemma~\ref{lemma:combinations} which leads to the exceptional cases.
\end{proof}

It turns out that the distances in $\M^2(T)$ are sufficient to completely describe the edges in every subsequent metamour graph of $T$.
The key to the proof of the following lemma is Corollary~\ref{cor:combinations} above, which will allow us always to split a path in $\M^2(T)$ into two subpaths of desired lengths.

\begin{lemma}
    \label{theorem: m-ary trees}
    Let $T = T(h,m)$ be the complete $m$-ary tree with height $h \geq 5$.
    Let $x,y \in V(T)$, with the exception of the case 
    \begin{equation}
        \label{exception}
        \operatorname{depth} x = \operatorname{depth} y = 1, \qquad m=2, \qquad h \in \{5,6\}.
    \end{equation}
    For $k \geq 2$, we have
   \[
    xy \in E(\M^k(T)) \Longleftrightarrow d_{\M^2(T)}(x,y) \in\quad \bigcup_{\mathclap{\substack{0\,\le\, i\, \leq\, k-2\\ i\, \equiv\, k \mod 2}}}\quad  S_i\,,
    \]
    where the sets $S_i$ are defined in~\eqref{Si}.
    In the case~\eqref{exception}, the edge $xy$ does not occur in $\M^k(T)$ for any $k \geq 2$.
\end{lemma}

\begin{proof}
    We use induction on $k$.
    In the base cases $k=2$ and $k=3$, the theorem is true by definition, since $S_0 = \{1\}$ and $S_1 = \{2\}$.
    Note that in the exceptional case~\eqref{exception}, we have $d_{\M^2(T)}(x,y) = 3$ (see the proof of Lemma~\ref{lemma:M3T}), and indeed $3$ does not belong to $S_0$ or $S_1$.

    As our induction hypothesis, assume that the theorem holds up to some value of $k$; we now show that it also holds for $k+1$.
    We first prove the ``$\Longrightarrow$'' direction of the biconditional in the theorem.
    For ease of notation, in the rest of this proof we abbreviate
    \[
    d \coloneqq d_{\M^2(T)}(x,y).
    \]
    Let $xy \in E(\M^{k+1}(T))$. Then $xy \notin E(\M^{k}(T))$. Thus, by our induction hypothesis, $d\notin \bigcup_{0\,\le\, i\, \leq\, k-2,\, i\, \equiv\, k \mod 2}  S_i$. Hence, we must have either $d\in \bigcup_{0\,\le\, i\, \leq\, k-3,\, i\, \equiv\, k+1 \mod 2}  S_i$ or $d\in \bigcup_{i\ge k-1}  S_i$.
    In the latter case, we claim that actually $d \in S_{k-1}$.
    To see this, recall that $xy \in E(\M^{k+1}(T))$ implies $xz,yz \in E(\M^{k}(T))$ for some $z\in V(T)$. 
    We have $d_{\M^2(T)}(x,z), d_{\M^2(T)}(z,y) \leq \max S_{k-2} = 2^{k-2}$ by induction hypothesis. 
    Hence, $d\le d_{\M^2(T)}(x,z)+d_{\M^2(T)}(z,y)\le 2^{k-1}$ by triangle inequality, which is the largest element of $S_{k-1}$.
    Thus, $xy \in E(\M^{k+1}(T))$ implies $d\in \bigcup_{0\,\le\, i\, \leq\, k-1,\, i\, \equiv\, k+1 \mod 2}  S_i$, which proves the ``$\Longrightarrow$'' direction in the theorem.

    To prove the converse, suppose that $d\in \bigcup_{0\,\le\, i\, \leq\, k-1,\, i\, \equiv\, k+1 \mod 2}  S_i$.
    Then automatically $xy \not\in E(\M^{k}(T))$ by induction hypothesis, and
    we must show that $xy \in E(\M^{k+1}(T))$.
    
    Assume for now that $d \neq 1,3$.
    Then, by Corollary~\ref{cor:combinations}, where $\ell = k-1$, the number $d$ can be written as the sum of two (possibly equal) elements
    \begin{equation}
    \label{bc}
    b,c \in \bigcup_{\mathclap{\substack{i\, \leq\, k-2 \\ i\, \equiv\, k \: {\rm mod} \: 2}}} S_i.
    \end{equation}
    Since $b+c = d$, there exists some $z \in V(T)$ such that $d_{M^2(T)}(x,z) = b$ and $d_{M^2(T)}(z,y) = c$.
    Moreover, by~\eqref{bc} and the induction hypothesis, both $xz$ and $yz$ lie in $E(\M^{k}(T))$.
    Since $xy$ is not an edge in $\M^{k}(T)$, we conclude that $xy \in E(\M^{k+1}(T))$, thereby proving the ``$\Longleftarrow$'' direction in the theorem (as long as $d \neq 1, 3$).
    
    It remains to treat the cases where $d\in \{1,3\}$.
    In either case, note that $k$ is odd.
    By Lemmas~\ref{lemma:4 equals 4 4s} and~\ref{lemma:M3T} (for $d=1$ and $d=3$, respectively), there exists a vertex $z \in V(T)$ such that both $xz$ and $yz$ are in $E(\M^3(T))$, except in the exceptional case~\eqref{exception}.
    Hence, outside of~\eqref{exception}, since $k$ is odd, it follows that both $xz$ and $yz$ are in $E(\M^k(T))$.
    Therefore we have $xy \in E(\M^{k+1}(T))$, which completes the proof.
    In the case~\eqref{exception}, by Lemma~\ref{lemma:M3T} there is no such vertex $z$, and so the edge $xy$ does not occur in $\M^4(T)$, nor in any successive metamour graph as can easily be verified by direct inspection.
    \end{proof}

We now give the main result of this section, showing that $T(h,m)$ has metamour limit period $2$, with a pre-period of $\lceil\log_2 h\rceil$:

\begin{theorem}
\label{cor:tree limits}
    Let $T = T(h,m)$ be the complete $m$-ary tree with height $h \geq 5$.
    Then we have $\M^k(T) = \M^{k+2}(T)$ if and only if $k \geq \lceil \log_2 h \rceil $.
    In this range, the two graphs in the metamour limit set of $T$ are given by the edge criterion
    \[
    xy \in E(\M^k(T)) \Longleftrightarrow \left\lceil \log_2 d_{\M^2(T)}(x,y) \right\rceil \equiv k \operatorname{mod} 2
    \]
    for all $x,y \in V(T)$, with the exception of the case where $\operatorname{depth} x = \operatorname{depth} y = 1$ with $m=2$ and $h \in \{5,6\}$; in that case, $xy$ is not an edge in either metamour limit graph.
\end{theorem}

\begin{proof}
    It is clear from Lemma~\ref{theorem: m-ary trees} that if $k \geq 2$, then $xy \in E(\M^k(T))$ implies $xy \in E(\M^{k+2}(T))$.
    By Lemma~\ref{lemma:diam M2T}, the maximum distance between two connected vertices in $\M^2(T)$ is $\lceil h/2 \rceil$.
    By~\eqref{Si log def}, the smallest integer $i$ such that $\lceil h/2 \rceil \in S_i$ is given by
    \begin{equation}
        \label{first Si}
    \Big\lceil \log_2 \lceil h/2 \rceil \Big\rceil = \lceil \log_2 h \rceil - 1.
    \end{equation}
    Therefore, if $d_{M^2(T)}(x,y) = \lceil h/2 \rceil$, then by Lemma~\ref{theorem: m-ary trees}, the quantity in~\eqref{first Si} is the unique value of $k$ such that $xy \not\in E(\M^k(T))$ but $xy \in E(\M^{k+2}(T))$. 
    Hence we have $\M^{k}(T) = \M^{k+2}(T)$ if and only if $k$ is strictly greater than~\eqref{first Si}.
\end{proof}

Finally, for the sake of completeness, we describe the metamour graphs in the cases where $h < 5$.
The following behavior can be verified directly by drawing the first few metamour graphs, and so we leave the details to the reader:

\begin{theorem}\

\begin{enumerate}
    \item Let $T = T(1,m)$.
    Then $\M(T) = K_m \cup K_1$, and for all $k \geq 2$, we have $\M^k(T) = \overline{K}_{m+1}$.

    \item Let $T = T(2,m)$.
    Then $\M(T) = K_m \cup {\rm Wd}(m,m)$, where ${\rm Wd}(m,m)$ is the windmill graph obtained by taking $m$ copies of $K_m$ with a common vertex.
    We then have $\M^2(T) = (\overline{K}_m)^{\nabla m} \cup \overline{K}_{m+1}$, followed by $\M^3(T) = (K_m)^{\cup\, m} \cup \overline{K}_{m+1}$.
    $($The notation $G^{\nabla m}$ denotes the join of $m$ copies of $G$.$)$
    For all $k \geq 4$, we have $\M^k(T) = \overline{K}_{m^2 + m + 1}$.
    \item Let $T = T(3,m)$.
    We have $\M^5(T) = (K_m)^{\cup\, m^2} \cup \overline{K}_{m^2 + m + 1}$, and so $\M^k(T) = \overline{K}_{m^3+m^2+m+1}$ for all $k \geq 6$.
    \item Let $T = T(4,m)$.
    Then $T$ has metamour limit period 2, with the metamour limit set as given in Theorem~\ref{cor:tree limits}.
    If $m=2$, then this periodic behavior begins at $k = 4$.
    If $m \geq 3$, then this periodic behavior begins at $k = 6$.
    
\end{enumerate}
    
\end{theorem}







\bibliographystyle{abbrvnat}
\bibliography{refs}

\end{document}